\documentclass[letterpaper,11pt]{amsart}
\usepackage{latexsym,array,delarray,amsthm,amssymb,epsfig,setspace,tikz,amsmath,enumerate,mathrsfs,graphicx, mathtools, comment}

\theoremstyle{plain}
\newtheorem{thm}{Theorem}[section]
\newtheorem{lemma}[thm]{Lemma}
\newtheorem{prop}[thm]{Proposition}
\newtheorem{cor}[thm]{Corollary}
\newtheorem{conj}[thm]{Conjecture}
\newtheorem*{thm*}{Theorem}
\newtheorem*{lemma*}{Lemma}
\newtheorem*{prop*}{Proposition}
\newtheorem*{cor*}{Corollary}
\newtheorem*{conj*}{Conjecture}

\theoremstyle{definition}
\newtheorem{defn}[thm]{Definition}
\newtheorem{ex}[thm]{Example}

\theoremstyle{remark}
\newtheorem*{rmk}{Remark}

\newtheoremstyle{case}{}{}{}{}{}{:}{ }{}
\theoremstyle{case}

\newcommand{\zz}{\mathbb{Z}}
\newcommand{\nn}{\mathbb{N}}

\newcommand{\rr}{\mathbb{R}}
\newcommand{\cc}{\mathbb{C}}

\newcommand{\cf}{\mathcal{F}}

\newcommand{\rmpa}{\mathrm{pa}}
\newcommand{\rmtop}{\mathrm{top}}

\newcommand{\ind}{\mbox{$\perp \kern-5.5pt \perp$}}

\newcommand{\indep}{\rotatebox[origin=c]{90}{$\models$}}

\newcommand\ddfrac[2]{\frac{\displaystyle #1}{\displaystyle #2}}

\begin{document}

\title[Directed Gaussian graphical models with toric vanishing ideals]{Directed Gaussian graphical models with toric vanishing ideals}
\author{Pratik Misra, Seth Sullivant }

\address{Department of Mathematics
\\North Carolina State University, Raleigh, NC, USA\\}

\email{ pmisra@ncsu.edu, smsulli2@ncsu.edu }

\keywords{clique sum, d-separation, t-separation, shortest trek map, safe gluing}

\begin{abstract}
Directed Gaussian graphical models are statistical models that use a directed acyclic graph (DAG) to represent the conditional independence structures between a set of jointly normal random variables. The DAG specifies the model through recursive factorization of the parametrization, via restricted conditional distributions. 
In this paper, we make an attempt to characterize the DAGs whose vanishing ideals
are toric ideals. In particular, we give some combinatorial criteria to construct such DAGs from smaller DAGs which have toric vanishing ideals. 
An associated monomial map called the \textit{shortest trek map} plays an
important role in our description of toric Gaussian DAG models. 
For DAGs whose vanishing ideal is toric, we prove results about the generating
sets of those toric ideals.
\end{abstract}

\maketitle

\section{Introduction}\label{sec:intro}

Gaussian graphical models are semi-algebraic subsets of the cone of positive definite covariance matrices. They are widely used throughout natural sciences, computational biology and many other fields \cite{Maathuis2019}. Graphical models can be defined by undirected graphs, directed acyclic graphs, or graphs that use a mixture of different types of edges. 
In this paper, we only consider those models which can be defined by 
directed acyclic graphs (DAGs). A DAG specifies a graphical model in two ways. 
The first way is via a combinatorial parametrization of covariance 
matrices that belong to the model and the second way is via 
conditional independence statements implied by the graph. 
The factorization theorem \cite[Thm 3.27]{Graphical Models} 
says that these two methods yield the same family of probability distribution functions.

The combinatorial parametrization of the covariance matrices for a Gaussian DAG model
is also known as the \textit{simple trek rule} (see e.g.~\cite{Trek Separation}). 
The vanishing ideal of the Gaussian DAG model, $I_G$,  is equal to the set of polynomials
in the covariances that are zero when evaluated at the simple
trek rule.   The algebraic interpretation of the second method, i.e., the conditional independence statements,  give us the conditional independence ideal $CI_G$. 
An important question that arises in the algebraic study of graphical models
is to determine the DAGs where the vanishing ideal and the conditional independence ideal 
are equal.  Although it is still an open problem, 
some past work and computational study \cite{Gaussian Networks} has been done in this direction.

The study of generators of the vanishing ideal $I_G$ is an important problem for constraint-based inference for inferring the structure of the underlying graph from data. For example, the TETRAD procedure \cite{Causation} specifically tests the degree 2 generators (tetrads) of the vanishing ideal for directed graphs to determine if the graphs have certain underlying features. 
In the undirected case, we showed  that the vanishing ideal is toric and is generated by polynomials of degree at most 2 if and only if $G$ is a 1-clique sum of complete graphs \cite{Block graphs}. 
Our goal in the present paper is to study the analogous problem for Gaussian DAG models.
While we are not able to give a complete characterization of the DAGs that have
degree two generators, and are toric, we develop methods to construct DAGs having toric vanishing ideals and understand the generating set of the vanishing ideal when it is toric. 
In particular, we develop three techniques to construct such DAGs with toric vanishing ideals
from smaller DAGs with the same property. 
These are called \textit{safe gluing, gluing at sinks} and \textit{adding a new sink.}

One of the important tools that we use throughout the paper is the \textit{shortest trek map} $\psi_G$.  We show that in some instances, 
the shortest trek map and the simple trek map have the same kernel, namely the ideal $I_G$.  
Being a monomial map, the kernel of $\psi_G$, which we denote by $ST_G$, 
is always a toric ideal.  Although $I_G, CI_G$ and $ST_G$ are not always equal, 
we are interested in characterizing the DAGs where these three ideals are the same. 
This not only tells us when the vanishing ideals are toric 
but we also get to know the structure of the generators of $I_G$ from $ST_G$.  
We show that when two DAGs $G_1$ and $G_2$ have toric vanishing ideals 
then gluing at sinks and adding a new sink always produces
a new graph $G$ with toric vanishing ideal.  
We also conjecture that the
same is true for the safe gluing of $G_1$ and $G_2$,
and prove a number of partial results towards this conjecture.
Further, we conjecture that every DAG whose vanishing ideal is 
toric can be obtained as a combination of these three operations 
starting with complete DAGs.

The paper is organized as follows.
Section \ref{sec:preliminaries}  gives an explicit description of the simple trek rule.
We recall  the notion of \textit{directed separation},
which is used in defining the conditional independence ideal.  
We also introduce the shortest trek map, and the shortest trek ideal $ST_G$.

In Section \ref{sec:safegluing} we look at some existing 
results from \cite{Graphical Models,Block graphs,Gaussian Networks}, 
about gluing graphs where nice properties of the vanishing ideals are preserved.
Using those results as inspiration, we construct a general operation which we call 
the ``\textit{safe gluing}" of DAGs.  Safe gluing is a type of clique sum for
DAGs such that most of the vertices in the clique are colliders along any paths passing through
the clique. 
We conjecture that when the vanishing ideals of two DAGs are the same 
as the kernel of their shortest trek maps, then a safe gluing 
of the two DAGs would also have a toric vanishing ideal. 
We prove this conjecture in some special cases.

In Section \ref{sec:sinks} we look at two more ways to construct 
new DAGs where the toric property is preserved, 
which we call \textit{gluing at sinks} and \textit{adding a new sink}. 
We analyze the generators of $ST_G$ in Section \ref{sec:shortesttrek} 
and show that the safe gluing action preserves the toric property when $ST_G$ equals $CI_G$ for the smaller DAGs, which further provides evidence for our Conjecture \ref{conj:safe gluing conjecture}.
In Section \ref{sec:conjectures}, 
we conclude with some conjectures which may be used to 
formulate a complete characterization of all possible DAGs having toric vanishing ideal.


\section{Preliminaries} \label{sec:preliminaries} 

This section primarily is concerned with preliminary definitions that we will
use throughout the paper.  We introduce the Gaussian DAG models, and their
vanishing ideals $I_G$.  We explain the concept of $d$-separation and
how this leads to the conditional independence ideal $CI_G$.  Finally,
we introduce the shortest trek map, and the corresponding shortest trek
ideal $ST_G$.

Computation of the vanishing ideal of a Gaussian undirected graphical model was first studied in \cite{BS n CU} and was later continued in \cite{Block graphs}. For DAGs, the computation of its vanishing ideal was introduced in \cite{Gaussian Networks} and was later seen in \cite{Trek Separation}. Prerequisites for these articles can be found in \cite{Algebraic Statistics}.

Let $G=(V,E)$ be a directed acyclic graph with vertex set $V(G)$ and edge set $E(G)$. As there are no directed cycles in the graph, we assume that the vertices are numerically ordered, i.e, $i \rightarrow j\in E(G)$ only if $i <j$. 
A \textit{parent} of a vertex $j$ is a node $i\in V(G)$ such that $i\rightarrow j$ is an edge in $G$. We denote the set of all parents of a vertex $j$ by $\rmpa(j)$. 
Given such a directed acyclic graph, we introduce a family of normal random variables that are related to each other by recursive regressions.

To each node $i$ in the graph, we introduce two random variables 
$X_i$ and $\varepsilon_i$. The $\varepsilon_i$ are independent normal variables 
$\varepsilon_i\sim \mathcal{N}(0,\omega_i)$ with $\omega_i>0$. For simplicity, we assume that all our random variables have mean zero. 
The recursive regression property of the DAG gives an expression for each $X_j$ in terms of $\varepsilon_j$, $X_i$ with $i<j$ and some regression parameters 
$\lambda_{ij} \in \rr$ assigned to the edges $i\rightarrow j$ in the graph,
\[
X_j=\sum_{i\in pa(j)} \lambda_{ij}X_i+\varepsilon_j.
\]
From this recursive sequence of regressions, we can solve for the covariance matrix 
$\Sigma$ of the jointly normal random vector $X$. 
This covariance matrix is given by a simple matrix factorization in terms of the regression parameters $\lambda_{ij}$ and the variance parameters $\omega_i$. Let $D$ be the diagonal matrix $D=$diag$(\omega_1,\omega_2,\ldots,\omega_m)$ and let $L$ be the $m\times m$ upper triangular matrix with $L_{ij}=\lambda_{ij}$ if $i\rightarrow j$ is an edge in $G$ and $L_{ij}=0$ otherwise.

\begin{prop}(\cite{Ancestral graph}, Section 8)\label{prop:parametrization}. The covariance matrix of the normal random variable $X=\mathcal{N}(0,\Sigma)$ is given by the matrix factorization 
\begin{equation}  \label{eq:matrixfactor}
\Sigma=(I-L)^{-T}D(I-L)^{-1}.
\end{equation}
\end{prop}

The vanishing ideal of the Gaussian graphical model is denoted by 
$I_G$ and it is an ideal in the polynomial ring 
$\cc[\Sigma] = \cc[\sigma_{ij} : 1 \leq i \leq j \leq n]$.
This is the ideal of all polynomials in the entries of the covariance matrix
$\Sigma$, that evaluate to zero for every choice of the parameters $\omega_i$ and
$\lambda_{ij}$.  That is:
\[
I_G =  \left\{  f \in \cc[\Sigma] :  f( (I-L)^{-T}D(I-L)^{-1}) = 0  \right\}. 
\]

One way to obtain $I_G$ is to eliminate the indeterminates $\omega_i$ and $\lambda_{ij}$ from the following system of equations:
\[
\Sigma-(I-L)^{-T}D(I-L)^{-1}=0.
\]
Using elimination is computationally expensive, and we are interested in theoretical
results that characterize the generators of $I_G$ when possible.

A variant on the parametrization (\ref{eq:matrixfactor}) is the  \textit{simple trek rule} which is a common and useful representation of the covariances in a Gaussian
DAG model. In order to explain the simple trek rule, we first need to go through a few definitions.
A \textit{collider} is a pair of edges $i\rightarrow k$, $j\rightarrow k$ with the same head. A \textit{path} in a DAG $G$ is a finite sequence of edges with any direction, which joins a sequence of vertices in $G$. If a path contains the edges $i\rightarrow k$ and $j\rightarrow k$, then the vertex $k$ is called the \textit{collider vertex} within that path. 
A path that does not repeat any vertex is called a \textit{simple path}. Let $T(i,j)$ be a collection of simple paths $P$ in $G$ from $i$ to $j$ such that there is no collider in $P$. Such a colliderless path is called a \textit{simple trek}. For the rest of the paper, we consider treks to be simple treks. We will often  use the notation $i \rightleftharpoons j$ to denote a specific trek between $i$ and $j$, as this helps to call attention to the endpoints.  When we speak generically of a trek,
we often denote it by $P$.

Each trek $P$ has a unique \textit{topmost} element  $\rmtop(P)$, which is the point where orientation of the path changes. A trek $P$ between $i$ and $j$ can also be represented by a pair of sets $(P_i,P_j)$, where $P_i$ corresponds to the directed path from $\rmtop(P)$ to $i$ and $P_j$ corresponds to the directed path from $\rmtop(P)$ to $j$. The vertex $\rmtop(P)$ is also called the \textit{common source} of $P_i$ and $P_j$. 

To get the simple trek rule, we introduce an alternate parameter $a_i$ associated to each node $i$ in the graph and is defined as the variance of $X_i$, i.e.~$\sigma_{ii}=a_i$. We expand the matrix product for $\Sigma$ in Proposition \ref{prop:parametrization} by taking the sum over all treks $P\in T(i,j)$. Using this expansion along with the alternate parameters $a_i$, we get the following definition :

\begin{defn}\label{defn:simple trek rule}
For a given DAG $G$, the \textit{simple trek rule} is defined as the rule in which the covariance $\sigma_{ij}$ is mapped to the sum of all possible simple treks from $i$ to $j$ in $G$. We represent the rule as a ring homomorphism $\phi_G$ where
\begin{eqnarray*}
\phi_G:\mathbb{C}[\sigma_{ij} : 1\leq i\leq j\leq n]&\rightarrow& \mathbb{C}[a_i,\lambda_{ij} : i,j\in [n], i\rightarrow j \in E(G)], \\
\sigma_{ij}&\mapsto& \sum_{P\in T(i,j)} a_{\rmtop(P)} \prod_{k\rightarrow l \in P}\lambda_{kl}. 
\end{eqnarray*}
\end{defn}

By Proposition 2.3 \cite{Gaussian Networks} we know that the kernel 
of the homomorphism $\phi_G$ equals the vanishing ideal 
$I_G$ of the model. We illustrate the simple trek rule with an example.

\begin{ex}\label{ex:phi_G}
Let $G_1$ be a directed graph on four vertices 
with edges $1\rightarrow 2,1\rightarrow 3, 1\rightarrow 4,2\rightarrow 3$ and $2\rightarrow 4$ 
(this is graph $G_1$ in Figure \ref{fig:shortest trek map}). The homomorphism $\phi_G$ is given by
\begin{eqnarray*}
\begin{split}
\sigma_{11}&\mapsto a_1 \\
\sigma_{12}&\mapsto a_1\lambda_{12} \\
\sigma_{13}&\mapsto a_1\lambda_{13} +a_1\lambda_{12}\lambda_{23} \\
\sigma_{14}&\mapsto a_1\lambda_{14}+a_1\lambda_{12}\lambda_{24} \\
\sigma_{22}&\mapsto a_2 \\
\end{split}
\quad \quad
\begin{split}
\sigma_{23}&\mapsto a_2\lambda_{23}+a_1\lambda_{12}\lambda_{13} \\
\sigma_{24}&\mapsto a_2\lambda_{24}+a_1\lambda_{12}\lambda_{14} \\
\sigma_{33}&\mapsto a_3 \\
\sigma_{34}&\mapsto a_2\lambda_{23}\lambda_{24}+a_1\lambda_{13}\lambda_{14} \\
\sigma_{44} &\mapsto a_4 \\
\end{split}
\end{eqnarray*}
The ideal $I_G$ is generated by a degree $3$ polynomial given by 
\[
I_G=\langle \sigma_{13}\sigma_{14}\sigma_{22}-\sigma_{12}\sigma_{14}\sigma_{23}-\sigma_{12}\sigma_{13}\sigma_{24}+\sigma_{11}\sigma_{23}\sigma_{24}+\sigma_{12}^2\sigma_{34}-\sigma_{11}\sigma_{22}\sigma_{34} \rangle.
\]
\end{ex}

We now look at the notion of \textit{directed separation} 
(also known as d-separation). The  d-separation criterion is used to construct the \textit{conditional independence} ideal $CI_G$.

\begin{defn}
Let $G$ be a DAG with $n$ vertices. Let $A$, $B$ and $C$ be disjoint subsets of $[n]$. Then $C$ \textit{d-separates} $A$ and $B$ if every path in $G$ connecting a vertex $i \in A$ to a vertex $j \in B$ contains a vertex $k$ that is either
\begin{enumerate}[i)]
    \item a non-collider that belongs to $C$ or
    \item a collider that does not belong to $C$ and has no descendants that belong to $C$. 
\end{enumerate}
\end{defn}

A key result for DAG models relates conditional independence to 
d-separation  (see e.g.~\cite[Sec. 3.2.2]{Graphical Models}).

\begin{prop} \label{prop:dsep}
The conditional independence statement $A\indep B|C$ holds for the directed Gaussian model associated to $G$ if and only if $C$ d-separates $A$ from $B$ in $G$. 
\end{prop}

Let $A$, $B$ and $C$ be disjoint subsets of $[n]$. The normal random vector
 $X\sim \mathcal{N}(\mu,\Sigma)$ satisfies the conditional independence constraint $A\indep B|C$ if and only if the submatrix $\Sigma_{A\cup C, B\cup C}$ has rank less than or equal to $|C|$. Combining this result with the definition of $d$-separation, we have the following:

\begin{defn}
The \textit{conditional independence ideal} of $G$ is defined as the ideal generated by the set of all $d$-separations in $G$, that is,
\[
CI_G=\langle (\#C+1) \text{ minors of } \Sigma_{A\cup C,B\cup C} | \hspace{1mm}C\hspace{1mm} d\text{-separates }A \text{ from } B \text{ in } G \rangle. 
\]
\end{defn}

Note that every covariance matrix in the Gaussian DAG model satisfies the 
conditional independence constraints obtained by d-separation.  
This means that $CI_G \subseteq I_G$. In fact, the variety of $CI_G$ defines
the model inside the cone of positive definite matrices.   Still, one would like to
understand when $CI_G = I_G$.  Towards this end, we study the related question
of when $I_G$ and, hence, $CI_G$ are toric.

\begin{ex}
Let $G$ be a DAG with 4 vertices as shown in Figure \ref{figure:Illustrating CI_G}. Observe that there exists no trek between the vertices $1$ and $2$ as the path $1\rightarrow 3 \leftarrow 2$ has a collider at $3$ and the other path $1 \rightarrow 3 \rightarrow 4 \leftarrow 2$ has a collider at $4$. So, we have that $\sigma_{12}\in CI_G$.

We now look at the two paths $1\rightarrow 3 \rightarrow 4$ and $1\rightarrow 3 \leftarrow 2 \rightarrow 4$ between $1$ and $4$. In the first path, $3$ is the only vertex in the path, which is also a non-collider. So, any $d$-separating set of $1$ and $4$ must contain $3$. But $\{3\}$ is not enough to $d$-separate $1$ from $4$ as $3$ is a collider vertex in the second path. So, we add the vertex $2$ to the $d$-separating set which gives us that $\{2,3\}$ $d$-separates $1$ from $4$. This implies that the $3\times 3$ minors of $\Sigma_{\{1,2,3\},\{2,3,4\}}\in CI_G$.

Computing $I_G$ and $CI_G$ gives us that 
\[
I_G=CI_G= \langle \sigma_{12}, \sigma_{12}\sigma_{23}\sigma_{34}-\sigma_{12}\sigma_{33}\sigma_{24}-\sigma_{13}\sigma_{22}\sigma_{34}+\sigma_{13}\sigma_{23}\sigma_{24}+\sigma_{14}\sigma_{22}\sigma_{33}-\sigma_{14}\sigma_{23}^2 \rangle,
\]
where the second generator of $CI_G$ is the determinant of $\Sigma_{\{1,2,3\},\{2,3,4\}}$.
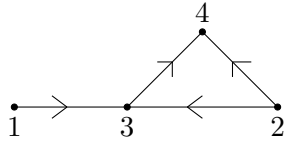
\begin{figure}
\begin{tikzpicture}
\filldraw[black]
(8.5,0) circle [radius=.04] node [below]
{1}
(10,0) circle [radius=.04] node [below]
{3}
(12,0) circle [radius=.04] node [below]
{2}
(11,1) circle [radius=.04] node [above]
{4};
\draw
(8.5,0)--++(3.5,0)
(10,0)--++(1,1)
(11,1)--++(1,-1);

\draw
(9,.15)--(9.2,0)
(9,-.15)--(9.2,0)
(10.4,0.6)--(10.6,.6)
(10.6,.4)--(10.6,.6)
(11,.15)--(10.8,0)
(11,-.15)--(10.8,0)
(11.4,0.4)--(11.4,.6)
(11.65,.6)--(11.4,.6);
\end{tikzpicture}
\caption{A DAG $G$ with 4 vertices\label{figure:Illustrating CI_G}}
\end{figure}
\end{ex}

To explain our results on when $I_G$ is toric,  we first give a brief description of toric ideals in general. Let $\mathcal{A}=\{a_1,a_2,\ldots,a_n\}$ be a fixed subset of $\zz^d$. We consider the homomorphism
\begin{eqnarray*}
\pi: \mathbb{N}^n\rightarrow \mathbb{Z}^d, \hspace{1cm} u=(u_1,\ldots,u_n)\mapsto u_1a_1+\cdots+u_na_n.
\end{eqnarray*}
This map $\pi$ lifts to a homomorphism of semigroup algebras:
\begin{eqnarray*}
\hat{\pi}: \cc[x_1,\ldots,x_n]\rightarrow \cc[t_1,\ldots,t_d,t_1^{-1},\ldots,t_d^{-1}], \hspace{.75cm}
x_i\mapsto t^{a_i}.
\end{eqnarray*}
The kernel of $\hat{\pi}$ is called the \textit{toric ideal} of $\mathcal{A}$,
and is denoted $I_\mathcal{A}$.  
By Lemma 4.1 of \cite{GB and CP} we know that the toric ideal can be generated by the (infinite) set of binomials of the form 
\begin{eqnarray*}
\{x^u-x^v : u,v \in \mathbb{N}^n \text{ with } \pi(u)=\pi(v) \}
\end{eqnarray*}
and so one question when confronted with a specific toric ideal is finding an
explicit finite set of binomial generators.  

From the construction above we observe that any monomial map can be written as 
$\hat{\pi}$ for some given set of vectors $\mathcal{A}$. 
This gives us that the kernel of every monomial map is a toric ideal. 
Further, every toric ideal is generated by a finite set of binomials. 
Because toric ideals are the vanishing ideals of monomial parametrizations, 
this leads us to identify a class of DAGs for which $I_G$ is obviously toric.

\begin{prop}\label{prop:simple trek-monomial map}
Let $G$ be a DAG such that there exists a unique simple trek (or no trek) between any two vertices of $G$. Then the simple trek rule is a monomial map hence $I_G$ is toric.
\end{prop}

\begin{proof}
As shown in Definition \ref{defn:simple trek rule}, 
the simple trek rule maps $\sigma_{ij}$ to the sum of all the treks between $i$ and $j$. So, if there exists a unique trek (or no trek) between any two vertices of $G$, then the simple trek rule becomes a monomial map and hence $I_G$ is toric. 
\end{proof}

Proposition \ref{prop:simple trek-monomial map} already shows that
the DAGs where $I_G$ is a toric ideal can be quite complicated.

\begin{ex}
Let $G$ be an undirected graph, and form a DAG by
replacing each undirected edge $i - j$ with two directed edges  $v_{i,j} \to i$ and $v_{i,j} \to j$, where $v_{i,j}$ is a new vertex.
The resulting DAG $\hat{G}$, has a unique simple trek between any pair of vertices,
or no trek, and so the ideal $I_{\hat{G}}$ is toric.   
\end{ex}

A second natural source of DAGs which have a toric vanishing ideal
are DAGs that have a natural connection to undirected graphs.
In previous work \cite{Block graphs}, we characterized the undirected 
Gaussian graphical models
which have toric vanishing ideals.

\begin{thm}\label{thm:toric undirected}[Theorem 1,\cite{Block graphs}]
The vanishing ideal $P_G$ of an undirected Gaussian graphical model is generated in degree $\leq 2$  if and only if each connected component of the graph 
$G$ is a $1$-clique sum of complete graphs (also known as \textit{block graphs}).
In this case, $P_G$ is a toric ideal.  
\end{thm}

Recall that a \emph{clique sum} of graphs $G_1$ and $G_2$ is a new graph obtained
by identifying two cliques of the same size in $G_1$ and $G_2$.  In a $k$-clique sum,
the cliques identified each have size $k$.  
While  Theorem \ref{thm:toric undirected} is a good starting point for
the analysis of DAG models, the underlying undirected structure is
not enough to characterize whether a DAG yields a toric vanishing ideal.

\begin{ex}\label{ex:3blocks}
Consider the three DAGs as given in Figure \ref{fig:3 block graphs}. Computing the vanishing ideals $I_{G_i}$, we get  
\begin{eqnarray*}
I_{G_1}&=&\langle \sigma_{12},\sigma_{13} \rangle \\
I_{G_2}&=& \langle \sigma_{12}\sigma_{23}-\sigma_{13}\sigma_{22},\sigma_{12}\sigma_{24}-\sigma_{14}\sigma_{22},\sigma_{13}\sigma_{24}-\sigma_{14}\sigma_{23} \rangle \\
I_{G_3}&=& \langle \sigma_{12},
\sigma_{12}\sigma_{23}\sigma_{34}-\sigma_{12}\sigma_{33}\sigma_{24}-\sigma_{13}\sigma_{22}\sigma_{34}+\sigma_{13}\sigma_{23}\sigma_{24}+\sigma_{14}\sigma_{22}\sigma_{33}-\sigma_{14}\sigma_{23}^2 \rangle.
\end{eqnarray*}
Note that all three DAGs have the same underlying undirected graph (up to permutation of labels), which is a $1$-clique sum of complete graphs. But only the first two DAGs have toric vanishing ideals. In $G_2$, the generators of $I_{G_2}$ correspond to the $2\times 2$ minors of $\Sigma_{12,234}$ as \{2\} $d$-separates \{1\} from \{3,4\}. Similarly, one of the generators of $I_{G_3}$ is the determinant of $\Sigma_{123,234}$ as \{2,3\} $d$-separates \{1\} from  \{4\}. Observe that the vertex \{3\} in $G_3$ is a collider within the path $1\rightarrow 3 \leftarrow 2\rightarrow 4$ and is a non collider within the trek $1\rightarrow 3\rightarrow 4$. 
This is an important observation for defining safe gluing later in the paper.

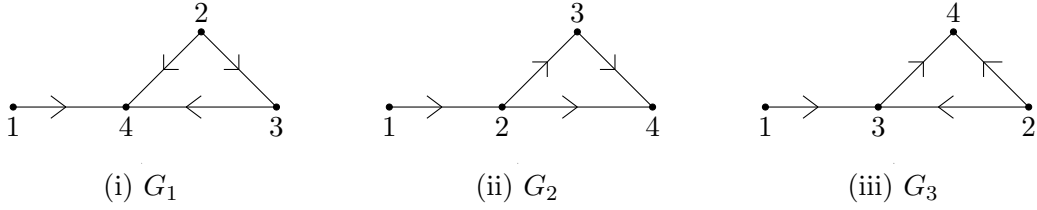
\begin{figure}
\begin{tikzpicture}
\filldraw[black]
(-1.5,0) circle [radius=.04] node [below] {1}
(0,0) circle [radius=.04] node [below] {4}
(1,1) circle [radius=.04] node [above] {2}
(2,0) circle [radius=.04] node [below] {3}
(.2,-.75) circle [radius=0] node [below] {(i) $G_1$};
\draw
(-1.5,0)--++(3.5,0)
(0,0)--++(1,1)
(1,1)--++(1,-1);
\draw 
(-1,.15)--(-.8,0)
(-1,-.15)--(-.8,0)
(0.5,0.7)--(.5,.5)
(0.7,.5)--(.5,.5)
(1,.15)--(.8,0)
(1,-.15)--(.8,0)
(1.5,0.7)--(1.5,.5)
(1.3,.5)--(1.5,.5);

\filldraw[black]
(3.5,0) circle [radius=.04] node [below] {1}
(5,0) circle [radius=.04] node [below] {2}
(7,0) circle [radius=.04] node [below] {4}
(6,1) circle [radius=.04] node [above] {3}
(5.2,-.75) circle [radius=0] node [below] {(ii) $G_2$}; 
\draw
(3.5,0)--++(3.5,0)
(5,0)--++(1,1)
(6,1)--++(1,-1);

\draw
(4,.15)--(4.2,0)
(4,-.15)--(4.2,0)
(5.4,0.6)--(5.6,.6)
(5.6,.4)--(5.6,.6)
(5.8,.15)--(6,0)
(5.8,-.15)--(6,0)
(6.5,0.7)--(6.5,.5)
(6.3,.5)--(6.5,.5);

\filldraw[black]
(8.5,0) circle [radius=.04] node [below]
{1}
(10,0) circle [radius=.04] node [below]
{3}
(12,0) circle [radius=.04] node [below]
{2}
(11,1) circle [radius=.04] node [above]
{4}
(10.2,-.75) circle [radius=0] node [below] {(iii) $G_3$}; 
\draw
(8.5,0)--++(3.5,0)
(10,0)--++(1,1)
(11,1)--++(1,-1);

\draw
(9,.15)--(9.2,0)
(9,-.15)--(9.2,0)
(10.4,0.6)--(10.6,.6)
(10.6,.4)--(10.6,.6)
(11,.15)--(10.8,0)
(11,-.15)--(10.8,0)
(11.4,0.4)--(11.4,.6)
(11.65,.6)--(11.4,.6);
\end{tikzpicture}
\caption{3 different DAGs having the same underlying undirected graph up to permutation of labels. \label{fig:3 block graphs}}
\end{figure}
\end{ex}

One thing that should be apparent in Example \ref{ex:3blocks} is that
the existence of a unique simple trek between pairs of vertices is not
a necessary condition for $I_G$ to be toric.  Indeed, in the DAG $G_1$,
there are two simple treks in $T(3,4)$ and yet the ideal $I_{G_1}$ is still
toric.  So in other cases when $I_G$ is toric, one way to demonstrate this
is to find an alternate parametrization for the ideal  $I_G$ that is monomial.
Our candidate for this new map is the 
 \textit{shortest trek map}.  This is defined  in a similar manner as the \textit{shortest path map} which played an important role in our proof
 of Theorem \ref{thm:toric undirected}.

\begin{defn}
Let $G$ be a DAG with $n$ vertices. Suppose that $G$ satisfies the
property that between any two vertices there is a unique shortest trek connecting them (or no trek connecting them).  For  vertices $i$ and $j$ in $G$, 
let $i\leftrightarrow j$ denote the shortest trek from $i$ to $j$ (if it exists). Then the \textit{shortest trek map} $\psi_G$ is given by
\begin{eqnarray*}
\psi_G: \mathbb{C}[\sigma_{ij}:1\leq i\leq j\leq n]&\rightarrow& \mathbb{C}[a_i,\sigma_{ij}:i,j\in [n],i\rightarrow j \in E(G)]\\
\psi_G(\sigma_{ij})&=&\begin{cases} 
          0 & \text{if there is no trek from }i \text{ to } j \\
          a_{\rmtop(i\leftrightarrow j)}\prod_{i'\rightarrow j' \in i \leftrightarrow j}\lambda_{i'j'} & \text{if shortest trek from } i \text{ to } j \text{ exists} \\
          a_i &  i=j . \\
          \end{cases}
\end{eqnarray*}
\end{defn}
The shortest trek map is defined only on those DAGs where there exists a unique shortest trek (or no trek) between any two vertices of $G$. We call the kernel of 
$\psi_G$ the \textit{shortest trek ideal} and denote it by $ST_G$.
We illustrate this with an example.

\begin{ex}\label{ex:shortest trek}

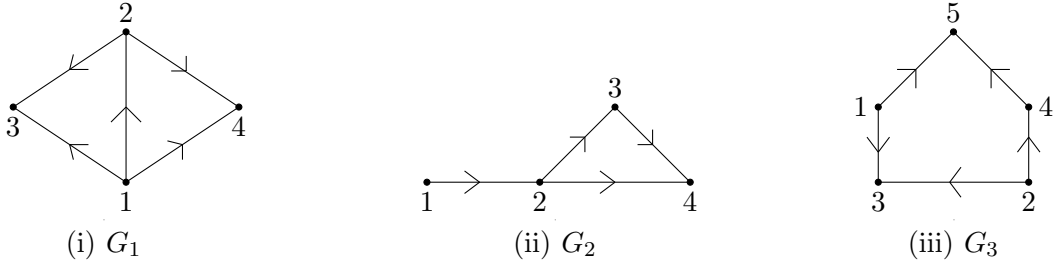
\begin{figure}
\begin{tikzpicture}
\filldraw[black]
(2,0) circle [radius=.04] node [below] {3}
(5,0) circle [radius=.04] node [below] {4}
(3.5,1) circle [radius=.04] node [above] {2}
(3.5,-1) circle [radius=.04] node [below] {1}
(3.2,-1.5) circle [radius=0] node [below] {(i) $G_1$};
\draw
(2,0)--(3.5,1)
(2,0)--(3.5,-1)
(5,0)--(3.5,-1)
(5,0)--(3.5,1)
(3.5,-1)--(3.5,1);
\draw 
(2.75,-.5)--(3,-.5)
(2.8,-.7)--(2.75,-.5)
(3.3,-.2)--(3.5,0)
(3.7,-.2)--(3.5,0)
(4.25,-.7)--(4.25,-.5)
(4.05,-.45)--(4.25,-.5)
(4.3,0.68)--(4.3,.48)
(4.1,.48)--(4.3,.48)
(2.75,.5)--(3,.5)
(2.8,.7)--(2.75,.5);

\filldraw[black]
(7.5,-1) circle [radius=.04] node [below] {1}
(9,-1) circle [radius=.04] node [below] {2}
(11,-1) circle [radius=.04] node [below] {4}
(10,0) circle [radius=.04] node [above] {3}
(9.2,-1.5) circle [radius=0] node [below] {(ii) $G_2$}; 
\draw
(7.5,-1)--++(3.5,0)
(9,-1)--++(1,1)
(10,0)--++(1,-1);

\draw
(8,-.85)--(8.2,-1)
(8,-1.15)--(8.2,-1)
(9.4,-.4)--(9.6,-.4)
(9.6,-.6)--(9.6,-.4)
(9.8,-.85)--(10,-1)
(9.8,-1.15)--(10,-1)
(10.5,-.3)--(10.5,-.5)
(10.3,-.5)--(10.5,-.5);

\filldraw[black]
(13.5,-1) circle [radius=.04] node [below] {3}
(13.5,0) circle [radius=.04] node [left] {1}
(14.5,1) circle [radius=.04] node [above] {5}
(15.5,0) circle [radius=.04] node [right] {4}
(15.5,-1) circle [radius=.04] node [below] {2}
(14.5,-1.5) circle [radius=0] node [below] {(iii) $G_3$};

\draw
(13.5,-1)--(13.5,0)
(13.5,0)--(14.5,1)
(14.5,1)--(15.5,0)
(15.5,0)--(15.5,-1)
(15.5,-1)--(13.5,-1);

\draw
(13.35,-.4)--(13.5,-.6)
(13.5,-.6)--(13.65,-.4)

(15.35,-.6)--(15.5,-.4)
(15.5,-.4)--(15.65,-.6)

(14.6,-.8)--(14.45,-1)
(14.45,-1)--(14.6,-1.2)

(13.75,.5)--(14,.5)
(14,.5)--(14,.25)

(15.25,.5)--(15,.5)
(15,.5)--(15,.25)
;
\end{tikzpicture}
\caption{Existence of a shortest trek map\label{fig:shortest trek map}}
\end{figure}

Let $G_1$ and $G_2$ be two DAGs as in Figure \ref{fig:shortest trek map}.
In $G_1$, there are exactly two treks of the same length from \{3\} to \{4\}. So, the shortest trek map is not defined for $G_1$.  
But as there exists a unique shortest trek between any two vertices in $G_2$, the shortest trek map $\psi_{G_2}$ is given by
\begin{eqnarray*}
\begin{split}
\sigma_{11}&\mapsto a_1 \\
\sigma_{12}&\mapsto a_1\lambda_{12} \\
\sigma_{13}&\mapsto a_1\lambda_{12}\lambda_{23} \\
\sigma_{14}&\mapsto a_1\lambda_{12}\lambda_{24} \\
\sigma_{22}&\mapsto a_2 \\
\end{split}
\quad \quad
\begin{split}
\sigma_{23}&\mapsto a_2\lambda_{23} \\
\sigma_{24}&\mapsto a_2\lambda_{24} \\
\sigma_{33}&\mapsto a_3 \\
\sigma_{34}&\mapsto a_2\lambda_{34} \\
\sigma_{44} &\mapsto a_4. \\
\end{split}
\end{eqnarray*}
Computing the vanishing ideal of $G_1$ gives us that $I_{G_1}$ is not toric as there exists a degree 3 minor in the  generating set (\{1,2\} $d$-separates \{3\} from \{4\}). But computing the kernel of the shortest trek map of $G_2$ gives us that
\[
\ker(\psi_{G_2})= ST_{G_2}=\langle \sigma_{12}\sigma_{23}-\sigma_{13}\sigma_{22},\sigma_{12}\sigma_{24}-\sigma_{14}\sigma_{22},\sigma_{13}\sigma_{24}-\sigma_{14}\sigma_{23} \rangle,
\]
which equals $I_{G_2}$ in Example \ref{ex:3blocks}. On the other hand, if we compute $ST_{G_3}$, we get
\[
ST_{G_3}= \langle \sigma_{14},\sigma_{12},\sigma_{13}\sigma_{15}-\sigma_{11}\sigma_{35},\sigma_{23}\sigma_{24}-\sigma_{22}\sigma_{34},\sigma_{24}\sigma_{45}-\sigma_{44}\sigma_{25} \rangle,
\]
which does not equal $I_{G_3}$ which has a generator of degree $3$ corresponding to a $3\times 3$ minor (as $\{1,2\}$ $d$-separates $\{3\}$ from $\{5\}$).
\end{ex}

In the example above we see that the shortest trek map does not exist for $G_1$. Although the existence of the shortest trek map does not ensure that $I_G$ would be toric (as seen in $G_3$), we do believe that $I_G$ cannot be toric when the shortest trek map is not well defined.  We look into this in more detail in Section \ref{sec:conjectures}.

The main problem of our interest is to find a characterization of the DAGs which have toric vanishing ideal and also understand the structure of its generators. In this context, it is also an important problem to understand when $I_G$ equals $CI_G$ as that would give us a definite structure of a generating set in terms of $d$-separations and minors. The ideal $ST_G$ comes into play here as we believe that $I_G$ is generated by monomials and binomials of degree at most $2$ if and only if $I_G$ is equal to $ST_G$. In the next two sections, we find ways to construct DAGs where $I_G = ST_G$.


\section{Safe gluing of DAGs }  \label{sec:safegluing}

As mentioned in the end of Section \ref{sec:preliminaries}, we are interested in those DAGs where $I_G$ equals $ST_G$. In this section we look at a specific way to construct such DAGs from smaller DAGs having the same property.
Given two DAGs $G_1$ and $G_2$ whose vanishing ideal is toric, there are various ways to glue $G_1$ and $G_2$ together. But the resultant DAG does not always have
a toric vanishing ideal. We are interested in those particular types of gluing operations
 which give us a toric vanishing ideal for the new DAG.
We use the term ``safe gluing" of two DAGs to denote a particular construction
which we conjecture to always preserve the toric property. 
Considering complete DAGs as the base case (as $I_G=0$ in that case), this method can be used to construct many DAGs which have toric vanishing ideal. 
The goal of this section is to explain the construction.
To motivate the concept of safe gluing, we first look at some existing results 
from the literature that give gluing operations on DAGs that preserve
the property of having a toric vanishing ideal.

\begin{defn}
Let $G$ be a DAG. A vertex $s$ in $G$ is called a \textit{sink} if all the edges adjacent to $s$ are directed towards $s$.
\end{defn}

\begin{prop}[Proposition 3.7, \cite{Gaussian Networks}]\label{prop:collider at m}
Let $G_1$ and $G_2$ be two DAGs having a common vertex $m$ that is a sink
in both $G_1$ and $G_2$. If $G$ is the new DAG obtained after 
gluing $G_1$ and $G_2$ at $m$, then $I_G$ can be written as
\[
I_G=I_{G_1}+I_{G_2}+\langle \sigma_{ij}  : i\in V(G_1)\setminus \{m\}, j\in V(G_2)\setminus \{m\} \rangle. 
\]
\end{prop}

The vertex $m$ in $G$ is a collider vertex within any path from $V(G_1)\setminus \{m\}$ to $V(G_2)\setminus \{m\}$. 
Further, if $I_{G_1}$ and $I_{G_2}$ are toric, then from Proposition \ref{prop:collider at m} we can conclude that gluing $G_1$ and $G_2$ at a vertex $m$ such that $m$ is a collider within any path from $V(G_1)\setminus \{m\}$ to $V(G_2)\setminus \{m\}$ produces a new DAG $G$ whose vanishing ideal is also toric.   In other words,
we have the following corollary.

\begin{cor}\label{cor:singlesinkglue}
Let $G_1$ and $G_2$ be two DAGs having a common vertex $m$ that is a sink
in both $G_1$ and $G_2$. Let  $G$ be the new DAG obtained after  gluing $G_1$ and $G_2$ at $m$.  If $I_{G_1}$ and $I_{G_2}$ are toric, then so is $I_G$.  Furthermore,
if  $I_{G_1} = ST_{G_1}$ and $I_{G_2} = ST_{G_2}$ then $I_G = ST_G$. 
\end{cor}

An example where this can be seen to occur is the graph $G_1$ in Example \ref{ex:3blocks}.  In $G_1$, we can consider the first graph as the subgraph with vertices $\{1,4\}$ and the second graph as the subgraph with vertices $\{2,3,4\}$. The vertex $4$ here is the common sink in both the graphs and hence the resultant vanishing ideal $I_{G_1}$ is toric.
We will generalize Corollary \ref{cor:singlesinkglue} in two ways. One is the safe gluing concept which is a combined generalization of Proposition \ref{prop:collider at m} and Corollary
\ref{cor:no collider}. The other is the concept of \textit{gluing at sinks} which we discuss in Section \ref{sec:sinks}.

A second situation where existing results in the literature can give us DAGS
with toric vanishing ideals concern situations where a DAG gives the same
independence structures as an undirected graph.  This is incapsulated in the 
concept of a perfect DAG.

\begin{defn}
Let $i,j,k$ be 3 vertices in a DAG $G$ containing the edges $i\rightarrow k$ and $j\rightarrow k$. Then $k$ is said to be an \textit{unshielded collider} in $G$ if $i$ and $j$ are not adjacent. A DAG $G$ is said to be \textit{perfect} if there are no unshielded colliders in $G$.
\end{defn}

Using the above definition, we state a result from \cite{Graphical Models}. 

\begin{prop}[Proposition 3.28, \cite{Graphical Models}]
Let $G$ be a perfect DAG and $G^\sim$ be its undirected version. Then the probability distribution $P$ admits a recursive factorization with respect to $G$ if and only if it factorizes according to $G^\sim$.
\end{prop}

In other words, the probability distributions coming from $G$ and $G^\sim$ are equal. 
In particular, this implies that 
\[ 
I_G = P_{G^\sim}
\]
for perfect DAGs (where $P_H$ denotes the vanishing ideal of Gaussian graphical model
associated to the undirected graph $H$).
On the other hand, we know from \cite{Block graphs} that in the undirected case, 
$P_H$ is toric if  $H$ is a 1-clique sum of complete graphs (also called a block graph). Hence, we have the following result :

\begin{cor}\label{cor:no collider}
Let $G$ be a DAG whose undirected version $G^\sim$ is a block graph. 
If $G$ is perfect then $I_G$ is toric.   
\end{cor}

We call a DAG $G$ where $G^\sim$ is a block graph and $G$ is perfect a 
\emph{perfect block DAG}.  Note that perfect block DAGs can be obtained by
gluing smaller perfect block DAGs together at a single vertex in such a way that 
no unshielded colliders are created.  

Corollaries \ref{cor:singlesinkglue} and \ref{cor:no collider}
give two different ways to glue DAGs together that have toric
vanishing ideals that preserve the toric property.  
Both methods consist of gluing the graphs at cliques of size one,
subject to some extra conditions.  
We generalize these criteria to obtain the safe gluing criteria in which a
DAG is obtained as an $N$-clique sum of two smaller DAGs so that the vanishing ideal is toric. To give the general definition of safe gluing, we first need to
recall the definition of a choke point.

\begin{defn}[Definition 4.1, \cite{Gaussian Networks}] A vertex $c\in V(G)$ is a 
\textit{choke point} between the sets $I$ and $J$ if every trek  from a 
vertex in $ I$ to a vertex in $J$ contains $c$ and either
\begin{enumerate}[i)]
    \item $c$ is on the $I$-side of every trek from $I$ to $J$, or
    \item $c$ is on the $J$-side of every trek from $I$ to $J$.
\end{enumerate}
\end{defn}

\begin{defn}
Let $G_1$ and $G_2$ be two DAGs.  Suppose that $G_1$ and $G_2$  share a common set of vertices 
$C = \{c \} \cup D$
 such that the induced subgraphs $G_1|_C$ and
$G_2|_C$ are the same and this common subgraph is a complete DAG (hence a clique).
The clique sum of $G_1$ and $G_2$ at $C$ is called a \emph{safe gluing} if
\begin{enumerate}[i)]
\item $c$ is a choke point between the sets 
$V(G_1)\setminus D$ and $V(G_2)\setminus D$ and
\item none of the treks between the vertices in $V(G_1)\setminus D$ 
and $V(G_2)\setminus D$ contain a vertex in $D$.
\end{enumerate}
\end{defn}

\begin{rmk}
Using the definition above, the gluing of $G_1$ and $G_2$ where there are no treks between the vertices of $V(G_1)\setminus C$ and $V(G_2)\setminus C$ can also be considered as a safe gluing.  Thus, both types of gluing operations implied by 
Corollaries \ref{cor:singlesinkglue} and \ref{cor:no collider} are safe
gluings.
\end{rmk}

We further illustrate the definition of safe gluing with an example.

\begin{ex}
Let $G_1$ and $G_2$ be two DAGs having a common $3$-clique at $\{1,4,5\}$ 
as shown in the figure \ref{fig:safe gluing at 3-clique}. 
Thus, $G$ is the DAG obtained after a safe gluing of $G_1$ and $G_2$ at the $3$-clique. 
Note that there is a single trek from $\{2\}$ to $\{3\}$ and that passes through 
$\{1\}$. Any other path from $\{2\}$ to $\{3\}$ containing $\{4\},\{5\}$ or 
both has a collider at $\{4\}$ or $\{5\}$. Hence, we can consider $D$ as the set $\{4,5\}$ and vertex $1$ as the choke point between the sets $\{1,2\}$ and $\{1,3\}$. Computing the vanishing ideal of 
$G$ gives us that $I_G=\langle \sigma_{12}\sigma_{13} -\sigma_{11}\sigma_{23} \rangle$, 
which is a toric ideal.

\begin{figure}
\begin{tikzpicture}
\filldraw[black]
(1,2) circle [radius=0]
(0,0) circle [radius=.06] node [left] {2}
(2,0) circle [radius=.06] node [below] {1}
(3,1.5) circle [radius=.06] node [above] {5}
(3,-1.5) circle [radius=.06] node [below] {4}
(2,-2) circle [radius=0] node [below] {$G_1$};
\draw
(0,0)--(2,0)
(0,0)--(3,1.5)
(0,0)--(3,-1.5)
(2,0)--(3,1.5)
(2,0)--(3,-1.5)
(3,1.5)--(3,-1.5);
\draw 
(1,0)--(1.2,-.2)
(1,0)--(1.2,.2)
(1.75,.9)--(1.4,.9)
(1.75,.9)--(1.65,.6)
(2.5,.7)--(2.5,.4)
(2.5,.7)--(2.2,.65)
(1.55,-.8)--(1.3,-.9)
(1.55,-.8)--(1.45,-.5)
(2.5,-.7)--(2.5,-.4)
(2.5,-.7)--(2.2,-.65)
(3,0)--(2.8,-.25)
(3,0)--(3.2,-.25);

\filldraw[black]
(5,0) circle [radius=.06] node [left] {1}
(6,1.5) circle [radius=.06] node [above] {5}
(6,-1.5) circle [radius=.06] node [below] {4}
(8,0) circle [radius=.06] node [right] {3}
(6.5,-2) circle [radius=0] node [below] {$G_2$};
\draw
(5,0)--(8,0)
(6,1.5)--(8,0)
(6,-1.5)--(8,0)
(5,0)--(6,1.5)
(5,0)--(6,-1.5)
(6,1.5)--(6,-1.5);
\draw 
(6.7,0)--(6.5,-.2)
(6.7,0)--(6.5,.2)
(6.8,.9)--(7.1,.9)
(6.8,.9)--(6.83,.6)
(5.5,.7)--(5.5,.4)
(5.5,.7)--(5.2,.65)
(6.8,-.9)--(7.1,-.9)
(6.8,-.9)--(6.83,-.6)
(5.5,-.7)--(5.5,-.4)
(5.5,-.7)--(5.2,-.65)
(6,-.1)--(5.8,-.35)
(6,-.1)--(6.2,-.35);

\filldraw[black]
(9.5,0) circle [radius=.06] node [left] {2}
(11.5,0) circle [radius=.06] node [below] {1}
(12.5,1.5) circle [radius=.06] node [above] {5}
(12.5,-1.5) circle [radius=.06] node [below] {4}
(14.5,0) circle [radius=.06] node [right] {3}
(12,-2) circle [radius=0] node [below] {$G$};
\draw
(9.5,0)--(14.5,0)
(9.5,0)--(12.5,1.5)
(9.5,0)--(12.5,-1.5)
(11.5,0)--(12.5,1.5)
(11.5,0)--(12.5,-1.5)
(12.5,1.5)--(12.5,-1.5)
(12.5,1.5)--(14.5,0)
(12.5,-1.5)--(14.5,0);

\draw 
(10.5,0)--(10.7,-.2)
(10.5,0)--(10.7,.2)
(11.25,.9)--(10.9,.9)
(11.25,.9)--(11.15,.6)
(12,.7)--(12,.4)
(12,.7)--(11.7,.65)
(11.05,-.8)--(10.8,-.9)
(11.05,-.8)--(10.95,-.5)
(12,-.7)--(12,-.4)
(12,-.7)--(11.7,-.65)
(12.5,-.1)--(12.3,-.35)
(12.5,-.1)--(12.7,-.35)
(13.3,.9)--(13.6,.9)
(13.3,.9)--(13.33,.6)
(13.3,-.9)--(13.6,-.9)
(13.3,-.9)--(13.33,-.6)
(13.2,0)--(13,-.2)
(13.2,0)--(13,.2);
\end{tikzpicture}
\caption{Safe gluing of $G_1$ and $G_2$ at a $3$-clique \label{fig:safe gluing at 3-clique}}
\end{figure}
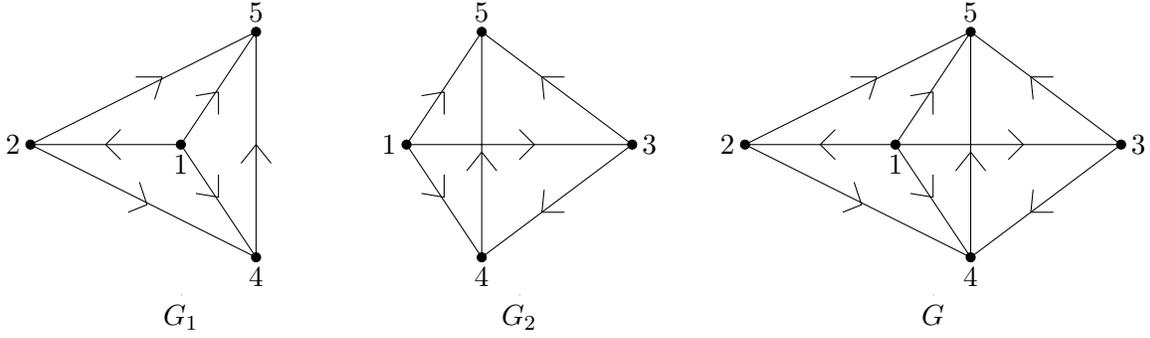
\end{ex}

We now look at some properties obtained from the safe gluing construction.

\begin{defn}
Let $G_1$ and $G_2$ be two DAGs, and suppose that $G$ is obtained from $G_1$ and $G_2$ by a safe
gluing at $C = \{c\} \cup D$.  This safe gluing is called a \emph{minimal safe gluing}
if we cannot find two other DAGs $G_1'$ and $G_2'$ such that $G$ is the safe gluing of
$G_1'$ and $G_2'$ at $\{c\} \cup D'$ with $D'$ a proper subset of $D$.
\end{defn}

\begin{ex}
Let $G$ be the DAG as shown in Figure \ref{figure:non minimal gluing}. If we take $G_1=\{1\rightarrow 3, 1\rightarrow 5, 3 \rightarrow 5\}$ and $G_2=\{2 \rightarrow 3, 2 \rightarrow 5, 3 \rightarrow 5, 4 \rightarrow 5\}$, then $G$ is a safe gluing of $G_1$ and $G_2$ with $C=\{3,5\}$. But this gluing is not minimal as we can take $G_1'=\{1\rightarrow 3, 1\rightarrow 5, 2 \rightarrow 3, 2 \rightarrow 5, 3 \rightarrow 5\}$ and $G_2'=\{4 \rightarrow 5\}$ such that $G$ is a safe gluing of $G_1'$ and $G_2'$ with $C'=\{5\} \subset C$.

\begin{figure}
\begin{tikzpicture}
\filldraw[black]
(0,0) circle [radius=.06] node [below] {$1$}
(2,0) circle [radius=.06] 
(2.3,0) circle [radius=0] node [below] {$5$}
(4,0) circle [radius=.06] node [below] {$2$}
(2,2) circle [radius=.06] node [above] {$3$}
(2,-2) circle [radius=.06] node [below] {$4$};

\draw
(0,0)--(4,0)
(0,0)--(2,2)
(2,2)--(2,-2)
(2,2)--(4,0);

\draw
(1.2,0)--(1,.2)
(1.2,0)--(1,-.2)
(2.8,0)--(3,.2)
(2.8,0)--(3,-.2)
(1.1,1.1)--(.8,1.1)
(1.1,1.1)--(1.1,.8)
(2.9,.8)--(2.9,1.1)
(2.9,1.1)--(3.2,1.1)
(2,.8)--(1.8,1)
(2,.8)--(2.2,1)
(2,-1)--(1.8,-1.2)
(2,-1)--(2.2,-1.2);
\end{tikzpicture}
\caption{Example of a non minimal gluing\label{figure:non minimal gluing}}
\end{figure}
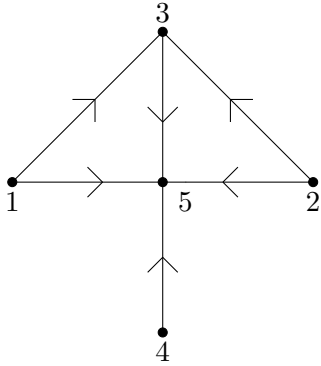
\end{ex}

One useful consequence of having a minimal safe gluing is that for any $d \in D$, there must exist a vertex $i \in V(G_1) \setminus C$ such that $i \rightarrow d$ is an edge (and analogously there is a $j \in V(G_2) \setminus C$). This is because if for some vertex $d\in D$ there does not exist any vertex in $V(G_1)\setminus C$ such that $i\rightarrow d$ is an edge, then it would mean that we can write $G$ as a safe gluing of $G_1'$ and $G_2$ at $C'=\{c\}\cup D\setminus \{d\}$ where $V(G_1')=V(G_1)\setminus \{d\}$. We use this observation for proving part $(ii)$ of Lemma \ref{lemma: n-1 clique}.

\begin{lemma}\label{lemma: n-1 clique}
Let $G_1$ and $G_2$ be two DAGs 
and 
$G$ be the resultant DAG obtained after a safe gluing of $G_1$ and $G_2$ at an 
$N$-clique. Let $C= \{c\} \cup D$ be the vertices in the $N$-clique.
\begin{enumerate}[i)]
    \item Every trek from a vertex in  
    $V(G_1)\setminus D$ to a vertex in  $V(G_2)\setminus D $ 
    must have the topmost vertex 
    (i.e, the source vertex) either always in  $G_1$ or always in $G_2$.
    \item If the gluing is minimal, then for each $d \in D$, we must have the edge $c \to d$.
    \end{enumerate}
\end{lemma}

\begin{proof}
\textbf{i)} To show this, let us assume that there are two treks 
$i_1\rightleftharpoons j_1$ and $i_2\rightleftharpoons j_2$ with 
$i_1,i_2\in V(G_1)\setminus D$ and $j_1,j_2\in V(G_2)\setminus D$ 
such that $\rmtop(i_1\rightleftharpoons j_1)$ lies in 
$V(G_1)\setminus D$ and $\rmtop(i_2\rightleftharpoons j_2)$ 
lies in $V(G_2)\setminus D$. 
Since $c$ must lie in these treks, since it is a choke point, 
this would imply that $c$ lies in the $G_1$-side of $i_1\rightleftharpoons j_1$ and the 
$G_2-$side of $i_2\rightleftharpoons j_2$.  That contradicts that  $c$ is a choke point. 

\textbf{ii)} Let us assume by way of contradiction
that $d\rightarrow c$ is an edge for some $d \in D$.
Since $G$ is obtained from a safe gluing, there are no edges that go from $d$ to any vertex in
$V(G_1) \setminus C$ or $V(G_2) \setminus C$.  For if there were such an edge $d \to i$, there would be a trek $c \leftarrow d \rightarrow i$ contradicting the definition of safe gluing.

Now as the gluing is minimal,
there must be vertices $s_1$ and $s_2$ in $G_1\setminus C$ and $G_2\setminus C$, respectively,
such that $s_1\rightarrow d$ and $s_2\rightarrow d$ are two edges in $E(G)$. 
By the definition of safe gluing, we know that $c$ must be a choke point between the sets 
$\{s_1,c\}$ and $\{s_2,c\}$. We consider the treks $s_1\rightleftharpoons c$ and 
$c\rightleftharpoons s_2$. 
As $s_1\rightarrow d\rightarrow c$ is already a trek, 
we cannot have any trek of the form 
$c\rightarrow t_1 \rightarrow t_2\rightarrow \cdots \rightarrow s_1$ 
(else it would form a cycle). 
So, $c$ always lies in the $G_2$-side of any trek $s_1\rightleftharpoons c$. 
Similarly, as $s_2\rightarrow d \rightarrow c$ is already a trek, we cannot have any trek of the form 
$c\rightarrow r_1\rightarrow r_2 \rightarrow \cdots \rightarrow s_2$. 
So, $c$ lies in the $G_1$-side of the treks $c_1\rightleftharpoons s_2$, which is a contradiction. 
\end{proof}

The observations in Lemma \ref{lemma: n-1 clique} 
are helpful for ruling out various bad scenarios 
as we work to prove results about the preservation of the toric property
for DAGs under safe gluing.

Our main aim in this section is to check that if 
$G_1$ and $G_2$ have toric vanishing ideals then a safe gluing of 
$G_1$ and $G_2$ would give us a DAG $G$ whose vanishing ideal is also toric. 
From the structure of $G$ we know that every trek between a vertex 
$i \in G_1\setminus C$ and $j \in G_2\setminus C$ passes through the choke point $c$. 
This allows us to decompose the treks $i\rightleftharpoons j$ as 
$i\rightleftharpoons c\cup c\rightleftharpoons j$. 
So, if we assume that $I_{G_1}=ST_{G_1}$ and $I_{G_2}=ST_{G_2}$, 
then this would imply that the shortest trek map is well defined for $G$ as well. Thus, we give the following conjecture :

\begin{conj}\label{conj:safe gluing conjecture}
Let $G_1$ and $G_2$ be two DAGs having toric vanishing ideals such that $I_{G_1}$ equals $ST_{G_1}$ and $I_{G_2}$ equals $ST_{G_2}$. If $G$ is the DAG obtained by a safe gluing of $G_1$ and $G_2$ at an $N$-clique, then $I_G$ is equal to $ST_G$ and hence is toric.
\end{conj}

Although we do not have a proof of Conjecture \ref{conj:safe gluing conjecture}, 
we provide a proof when $I_{G_1}$ and $I_{G_2}$ satisfy an extra condition.

\begin{thm}\label{thm:safe gluing at n clique}
Let $G_1$ and $G_2$ be two DAGs such that $I_{G_1}$ equals $ST_{G_1}$ and $I_{G_2}$ equals $ST_{G_2}$. 
Let $G$ be the DAG obtained by a safe gluing of $G_1$ and $G_2$ at an $N$-clique and $c$ be the choke point. 
If the generators of $I_{G_1}$ and $I_{G_2}$ have at most one common indeterminate $\sigma_{cc}$,
then $I_G$ is equal to $ST_G$ and hence is toric.
\end{thm}

\begin{proof}
Let $C= \{c \}\cup D$ be the $N$-clique where $G_1$ and $G_2$ are glued. We break the problem into two
cases:  The first case is when the vertex $c$ lies on some treks from $V(G_1) \setminus C$ to
$V(G_2) \setminus C$.  The second case is when there are no such treks.

\textbf{Case I} : The choke point $c\in C$ is on some trek from $V(G_1) \setminus C$ to
$V(G_2) \setminus C$.  In particular, it will be a non-collider vertex along that path.

As $c$ is the only vertex in $C$ that can be on some trek from $V(G_1) \setminus C$ to
$V(G_2) \setminus C$, 
no trek between any two vertices in $V(G_1)\setminus C$ passes through a vertex in $V(G_2)\setminus C$ 
(and similarly similarly for vertices in $V(G_2)\setminus C$). 
Further, $ST_{G_1}$ equals $I_{G_1}$, which implies that there exists a 
unique shortest trek (or no trek) between any two vertices in $G_1$ (similarly for $G_2$). 
Now, from the structure of $G$ we know that every trek between a vertex in 
$V(G_1)\setminus C$ and 
$V(G_2)\setminus C$ must pass through $c$. 
So, we can write the shortest trek map of $G$ as follows :
\begin{eqnarray*}
\psi_G(\sigma_{ij})= \begin{cases} \hspace{1cm}\psi_{G_1}(\sigma_{ij}) \hspace{1cm} : i,j \in V(G_1) \\
\hspace{1cm}\psi_{G_2}(\sigma_{ij}) \hspace{1cm} : i, j \in V(G_2)  \\
\ddfrac{\psi_{G_1}(\sigma_{ic}).\psi_{G_2}(\sigma_{c_1j})}{a_{c}} \hspace{.4cm} : i \in V(G_1)\setminus C, j\in V(G_2)\setminus C .
\end{cases} 
\end{eqnarray*}

Also, we know that the conditional independence statement $i \indep j | c$ holds for all 
$i \in V(G_1)\setminus C$ and $j\in V(G_2) \setminus C$. 
So $\sigma_{ic}\sigma_{cj}-\sigma_{ij}\sigma_{cc} $ lies in both 
$I_G$ and $ST_G$ for all $i \in V(G_1)\setminus C$ and $j\in V(G_2)\setminus C$. 

The vanishing ideals $I_{G_1}$ and $I_{G_2}$ lie in the polynomial rings 
$\mathbb{C}[\sigma_{ij}: i,j \in V(G_1)]$ and $\mathbb{C}[\sigma_{ij}: i,j \in V(G_2)]$ respectively, 
where the common indeterminates are of the form $\sigma_{c_ic_j}, c_i,c_j\in C$. 
But from the assumption, we know that $\sigma_{cc}$ can be the only common indeterminate 
among the generators of $I_{G_1}$ and $I_{G_2}$. 
So, without loss of generality, we can treat the ideals $I_{G_1}$ and $I_{G_2}$ as if they
lie in different rings, that contain enough indeterminates for all their generators.  In particular,
we can treat the ideals as if they belong to:
\begin{eqnarray*}
I_{G_1}&\subseteq& \mathbb{C}[\sigma_{ij}: i,j \in V(G_1)\setminus D] \text{ and} \\
I_{G_2}&\subseteq& \mathbb{C}[\sigma_{ij}: i,j \in V(G_2)].
\end{eqnarray*}

Note that there is only the indeterminate $\sigma_{cc}$ common between the two rings $\mathbb{C}[\sigma_{ij}: i,j \in V(G_1)\setminus D] $ and $\mathbb{C}[\sigma_{ij}: i,j \in V(G_2)]$.

Now, let $f=\sigma^u-\sigma^v$ be any binomial in a generating set of $ST_G$ consisting of primitive binomials. 
Suppose that $i\in V(G_1)\setminus C$ and $j\in V(G_2)\setminus C$.  We can
replace $\sigma_{ij}$ with $\frac{\sigma_{ic}\sigma_{cj}}{\sigma_{cc}}$  in both $\sigma^u$ and $\sigma^v$. 
Multiplying enough powers of $\sigma_{cc}$, we get
\[
\sigma_{cc}^r f=\sigma^{u_1}\sigma^{u_2}-\sigma^{v_1}\sigma^{v_2} \sigma_{cc}^m,
\]
(modulo the quadratic generators $\sigma_{ic}\sigma_{cj}-\sigma_{ij}\sigma_{cc} $ that
belong to $I_G$), 
where $\sigma^{u_1},\sigma^{v_1}\in \mathbb{C}[\sigma_{ij}: i,j \in V(G_1)\setminus D]$ 
and $\sigma^{u_2},\sigma^{v_2} \in \mathbb{C}[\sigma_{ij}:i,j \in V(G_2)]$, but none
of $\sigma^{u_1},\sigma^{v_1}, \sigma^{u_2},\sigma^{v_2}$ involve the indeterminate $\sigma_{cc}$.

We can split the monomial $\sigma_{cc}^m =  \sigma_{cc}^{m_1} \sigma_{cc}^{m_2}$ so that the two
binomials 
\[
\sigma^{u_1} - \sigma^{v_1} \sigma_{cc}^{m_1}  \quad \quad \mbox{ and } \quad \quad 
\sigma^{u_2} - \sigma^{v_2} \sigma_{cc}^{m_2}
\]
are homogeneous.  Since all the indeterminates appearing in $\sigma^{u_1}$  and $\sigma^{v_1}$ 
involve parameters from the graph $G_1$ with no overlap with parameters from $G_2$ (except possibly $a_{cc}$)
we see that if $\sigma^{u_1}\sigma^{u_2}-\sigma^{v_1}\sigma^{v_2} \sigma_{cc}^m$ belongs to $ST_G$,
it must be the case that $\sigma^{u_1} - \sigma^{v_1} \sigma_{cc}^{m_1}$ belongs to $ST_{G_1}$.
Then if $\sigma^{u_1}\sigma^{u_2}-\sigma^{v_1}\sigma^{v_2} \sigma_{cc}^m$ is to belong to 
$ST_G$, then it must also be the case that $\sigma^{u_2} - \sigma^{v_2} \sigma_{cc}^{m_2}$ belongs to 
$ST_{G_2}$.  

Now, modulo the quadratic generators $\sigma_{ic}\sigma_{cj}-\sigma_{ij}\sigma_{cc} $ that
belong to $I_G$, we have that
\begin{eqnarray*}
\sigma_{cc}^r f &=& \sigma^{u_1}(\sigma^{u_2}-\sigma_{cc}^{m_2}\sigma^{v_2})
                  + \sigma_{cc}^{m_2}\sigma^{v_2}(\sigma^{u_1}-\sigma_{cc}^{m_1}\sigma^{v_1}) \\
&\in& ST_{G_1}+ ST_{G_2} \\
&=& I_{G_1}+I_{G_2} \subseteq I_G.
\end{eqnarray*}

Thus, $\sigma_{cc}^r f \in I_G$.  As $I_G$ is a prime ideal, that does not contain $\sigma_{cc}$, we deduce
that $f \in I_G$.   This implies that $ST_G\subseteq I_G$. 
The vanishing ideal $I_G$ is well-known to have dimension $n + e$, as the model is
identifiable.  The dimension of $ST_G$ equals $n+e$ by
Proposition \ref{prop:dim of STG}.
But as the dimension of $I_G$ equals the dimension of $ST_G$,
both ideals are prime, and $ST_G\subseteq I_G$, 
we can conclude that $I_G= ST_G$.

\textbf{Case II}: 
There are no treks between the vertices of $V(G_1)\setminus C$ and $V(G_2)\setminus C$ : 
In this case, the shortest trek map $\psi_G$ can be written as :

\begin{eqnarray*}
\psi_G(\sigma_{ij})= \begin{cases} \psi_{G_1}(\sigma_{ij}) :  i,j \in V(G_1) \\
\psi_{G_2}(\sigma_{ij}) :  i, j \in V(G_2)  \\
\hspace{.4cm} 0 \hspace{.8cm}:  i \in V(G_1)\setminus C, j\in V(G_2)\setminus C .
\end{cases} 
\end{eqnarray*}

We claim that $ST_G$ in this case is  
\[
ST_G= ST_{G_1}+ ST_{G_2}+\langle \sigma_{ij}:i\in V(G_1)\setminus C, j\in V(G_2)\setminus C \rangle . 
\]

We prove this equality in the same way as the proof of Proposition \ref{prop:collider at m}. 
We have $\psi_G(\sigma_{ij})=0$ for all $i\in V(G_1)\setminus C$ and $j\in V(G_2)\setminus C$. 
By our assumption we know that none of the indeterminates of the form 
$\sigma_{cd}$ or $\sigma_{d, d'},$ with $d, d' \in D$ can appear among any of the generators of 
$ST_{G_1}$. 
Also in this case, $\sigma_{cc}$ cannot appear in 
$ST_{G_1}$ or $ST_{G_2}$ as $\psi_{G_1}(\sigma_{cc})=\psi_{G_2}(\sigma_{cc})=a_{c}$ 
and no treks involving $i \in V(G_1) \setminus C$ or $j \in V(G_2) \setminus C$ can have $c$ as its source. 
So, for any $\sigma_{ij}, i\in V(G_1)\setminus C, j\in V(G_1)$ and 
$\sigma_{kl},k \in V(G_2),l\in V(G_2)\setminus \{c\}$ which appear in $ST_G$, 
$\psi_G(\sigma_{ij})$ and $\psi_G(\sigma_{kl})$ are monomials in two polynomial rings 
having disjoint indeterminates. Thus, we have a partition of the indeterminates $\sigma_{ij}$ into three sets 
\begin{eqnarray*}
A_1&=& \{\sigma_{ij}: i\in V(G_1), j\in V(G_1)\setminus C\}, \\
A_2&=& \{\sigma_{ij}: i,j\in V(G_2)\} \text{ and} \\ 
A_3&=& \{\sigma_{ij}: i\in V(G_1)\setminus C,j\in V(G_2)\setminus C \},
\end{eqnarray*}
in which the image $\psi_G(\sigma_{ij})$ appears in disjoint sets of indeterminates. Further, there can be no nontrivial relations involving two or more of these three sets of indeterminates. So, the equality in the above equation holds. 

But then, $ST_{G_1}=I_{G_1}$ and $ST_{G_2}=I_{G_2}$. Thus, we have  
\begin{eqnarray*}
ST_G &=&I_{G_1}+I_{G_2}+ \langle \sigma_{ij}:i\in V(G_1)\setminus C, j\in V(G_2)\setminus C \rangle \\
&\subseteq&  I_G. 
\end{eqnarray*}
As both the ideals are prime and have the same dimension, $ST_G=I_G$.
\end{proof}

Although Theorem \ref{thm:safe gluing at n clique} uses the assumption that only
$\sigma_{cc}$ appears among the generators of both $I_{G_1}$ and $I_{G_2}$, 
we believe that the safe gluing would yield a toric vanishing ideal even without that assumption. 
We illustrate this point with an example.

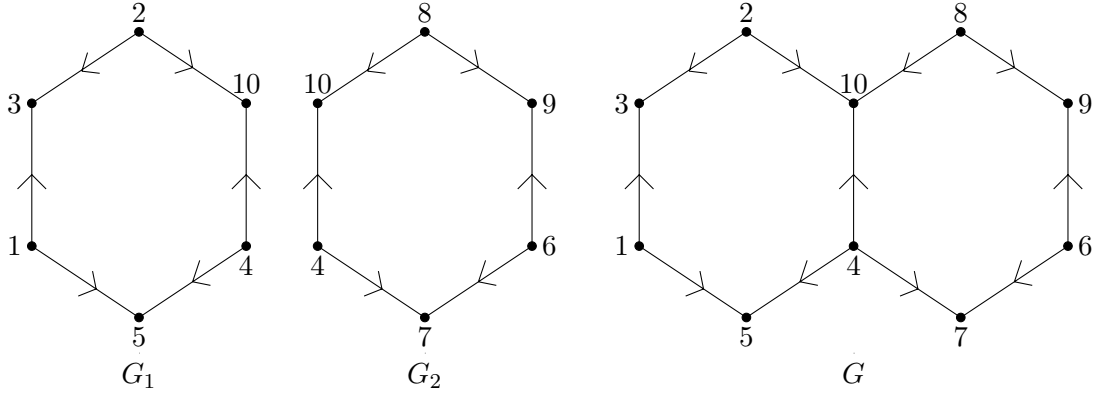
\begin{figure}
\begin{tikzpicture}[scale=0.95]
\filldraw[black]
(0,0) circle [radius=.06] node [left] {1}
(0,2) circle [radius=.06] node [left] {3}
(1.5,-1) circle [radius=.06] node [below] {5}
(1.5,3) circle [radius=.06] node [above] {2}
(3,0) circle [radius=.06] node [below] {4}
(3,2) circle [radius=.06] node [above] {10}
(1.5,-1.5) circle [radius=0] node [below] {$G_1$};
\draw
(0,0)--(0,2)
(0,0)--(1.5,-1)
(0,2)--(1.5,3)
(1.5,-1)--(3,0)
(3,0)--(3,2)
(1.5,3)--(3,2)

(-.2,.8)--(0,1)
(.2,.8)--(0,1)

(.95,2.42)--(0.7,2.45)
(.7,2.45)--(.78,2.71)

(2.25,2.5)--(2.2,2.75)
(2.25,2.5)--(2,2.42)

(.9,-.6)--(.65,-.7)
(.9,-.6)--(.85,-.35)

(2.25,-.5)--(2.35,-.2)
(2.25,-.5)--(2.5,-.55)

(2.8,.8)--(3,1)
(3.2,.8)--(3,1);

\filldraw[black]
(4,0) circle [radius=.06] node [below] {4}
(4,2) circle [radius=.06] node [above] {10}
(5.5,-1) circle [radius=.06] node [below] {7}
(5.5,3) circle [radius=.06] node [above] {8}
(7,0) circle [radius=.06] node [right] {6}
(7,2) circle [radius=.06] node [right] {9}
(5.5,-1.5) circle [radius=0] node [below] {$G_2$};
\draw
(4,0)--(4,2)
(4,0)--(5.5,-1)
(4,2)--(5.5,3)
(5.5,-1)--(7,0)
(7,0)--(7,2)
(5.5,3)--(7,2)

(3.8,.8)--(4,1)
(4.2,.8)--(4,1)

(4.95,2.42)--(4.7,2.45)
(4.7,2.45)--(4.78,2.71)

(6.25,2.5)--(6.2,2.75)
(6.25,2.5)--(6,2.42)

(4.9,-.6)--(4.65,-.7)
(4.9,-.6)--(4.85,-.35)

(6.25,-.5)--(6.35,-.2)
(6.25,-.5)--(6.5,-.55)

(6.8,.8)--(7,1)
(7.2,.8)--(7,1);

\filldraw[black]
(8.5,0) circle [radius=.06] node [left] {1}
(8.5,2) circle [radius=.06] node [left] {3}
(10,-1) circle [radius=.06] node [below] {5}
(10,3) circle [radius=.06] node [above] {2}
(11.5,0) circle [radius=.06] node [below] {4}
(11.5,2) circle [radius=.06] node [above] {10}
(13,-1) circle [radius=.06] node [below] {7}
(13,3) circle [radius=.06] node [above] {8}
(14.5,0) circle [radius=.06] node [right] {6}
(14.5,2) circle [radius=.06] node [right] {9}
(11.5,-1.5) circle [radius=0] node [below] {$G$};
\draw
(8.5,0)--(8.5,2)
(8.5,0)--(10,-1)
(8.5,2)--(10,3)
(10,-1)--(11.5,0)
(11.5,0)--(11.5,2)
(10,3)--(11.5,2)
(11.5,2)--(13,3)
(11.5,0)--(13,-1)
(13,-1)--(14.5,0)
(14.5,0)--(14.5,2)
(13,3)--(14.5,2)

(8.3,.8)--(8.5,1)
(8.7,.8)--(8.5,1)

(9.45,2.42)--(9.2,2.45)
(9.2,2.45)--(9.28,2.71)

(10.75,2.5)--(10.7,2.75)
(10.75,2.5)--(10.5,2.42)

(9.4,-.6)--(9.15,-.7)
(9.4,-.6)--(9.35,-.35)

(10.75,-.5)--(10.85,-.2)
(10.75,-.5)--(11,-.55)

(11.3,.8)--(11.5,1)
(11.7,.8)--(11.5,1)

(12.45,2.42)--(12.2,2.45)
(12.2,2.45)--(12.28,2.71)

(13.75,2.5)--(13.7,2.75)
(13.75,2.5)--(13.5,2.42)

(12.4,-.6)--(12.15,-.7)
(12.4,-.6)--(12.35,-.35)

(13.75,-.5)--(13.85,-.2)
(13.75,-.5)--(14,-.55)

(14.3,.8)--(14.5,1)
(14.7,.8)--(14.5,1);;
\end{tikzpicture}
\caption{Safe gluing of $G_1$ and $G_2$ at a $2$-clique \label{figure:non chordal gluing}}

\end{figure}

\begin{ex}\label{ex:non chordal safe gluing}
Let $G_1$ and $G_2$ be two non chordal cycles as shown in Figure \ref{figure:non chordal gluing}. Computing the vanishing ideals $I_{G_1}$ and $I_{G_2}$, we get 
\begin{eqnarray*}
I_{G_1}&=& \langle \sigma_{24},\sigma_{14},\sigma_{12},\sigma_{1,10},\sigma_{25},\sigma_{34},\sigma_{23}\sigma_{2,10}-\sigma_{22}\sigma_{3,10},\sigma_{13}\sigma_{15}-\sigma_{11}\sigma_{35},\sigma_{45}\sigma_{4,10}-\sigma_{44}\sigma_{5,10}
\rangle, \\
I_{G_2}&=& \langle \sigma_{6,10},\sigma_{78},\sigma_{68},\sigma_{49},\sigma_{48},\sigma_{46},\sigma_{89}\sigma_{8,10}-\sigma_{88}\sigma_{9,10},\sigma_{67}\sigma_{69}-\sigma_{66}\sigma_{79},\sigma_{47}\sigma_{4,10}-\sigma_{44}\sigma_{7,10} \rangle,
\end{eqnarray*}
which are both toric ideals. Now, if we perform a safe gluing of $G_1$ and $G_2$ at the $2$-clique $C=\{4,10\}$, we get the resultant DAG $G$ as in the figure. Observe that the indeterminate $\sigma_{4,10}$ appears in the vanishing ideal of both $G_1$ and $G_2$. Computing the vanishing ideal $I_G$ gives us
\begin{eqnarray*}
I_G &=& \langle \sigma_{14},\sigma_{12},\sigma_{6,10},\sigma_{68},\sigma_{49},\sigma_{48},\sigma_{29},\sigma_{46},\sigma_{28},\sigma_{27},\sigma_{26},\sigma_{25},\sigma_{24},\sigma_{78},\sigma_{59},\sigma_{58},\sigma_{39},\sigma_{56},\sigma_{38},\\
&& \sigma_{1,10},\sigma_{19},\sigma_{37},\sigma_{18},\sigma_{36},\sigma_{17},\sigma_{16},\sigma_{34},\sigma_{13}\sigma_{15}-\sigma_{11}\sigma_{35},\sigma_{89}\sigma_{8,10}-\sigma_{88}\sigma_{9,10},\\
&& \sigma_{67}\sigma_{69}-\sigma_{66}\sigma_{79}, \sigma_{4,10}\sigma_{57}-\sigma_{47}\sigma_{5,10},\sigma_{4,10}\sigma_{57}-\sigma_{45}\sigma_{7,10},\sigma_{47}\sigma_{4,10}-\sigma_{44}\sigma_{7,10},\\
&& \sigma_{45}\sigma_{4,10}-\sigma_{44}\sigma_{5,10},\sigma_{45}\sigma_{47}-\sigma_{44}\sigma_{57},\sigma_{23}\sigma_{2,10}-\sigma_{22}\sigma_{3,10}
\rangle \\
&& = I_{G_1} + I_{G_2} + \langle \sigma_{ij} : i \in V(G_1)\setminus C, j \in V(G_2) \setminus C \rangle = ST_G,
\end{eqnarray*}
which is still a toric ideal.
\end{ex}


\section{Gluing at sinks and adding a new sink} \label{sec:sinks}
We now look at two more ways of constructing DAGs which have 
toric vanishing ideals.  Both methods involve sinks in the DAGs.
The first construction we analyze is gluing the two graphs together at the sinks.
The second concept involves adding new sinks to the DAG.

\begin{defn}
Let $G_1$ and $G_2$ be two DAGs and $S_1, S_2$ be the set of sinks in $G_1$ and $G_2$ respectively. 
If $S$ is the set of all the common vertices in $S_1$ and $S_2$,
then \textit{gluing $G_1$ and $G_2$ at the sinks} 
refers to the construction of a new DAG $G$ with vertex set $V(G_1)\cup V(G_2)$ and edge set $E(G_1)\cup E(G_2)$. 
\end{defn}
We illustrate this construction of gluing at sinks with an example.

\begin{figure}
\begin{tikzpicture}[scale=0.6]
\filldraw[black]
(0,0) circle [radius=.06] node [left] {7}
(1.5,1.5) circle [radius=.06] node [above] {1}
(3,0) circle [radius=.06] node [below] {8}
(4.5,1.5) circle [radius=.06] node [above] {2}
(6,0) circle [radius=.06] node [below] {9}
(7.5,1.5) circle [radius=.06] node [above] {3}
(9,0) circle [radius=0.06] node [right] {10}
(4.5,-.5) circle [radius=0] node [below] {$G_1$};

\draw
(0,0)--(1.5,1.5)
(1.5,1.5)--(3,0)
(3,0)--(4.5,1.5)
(4.5,1.5)--(6,0)
(6,0)--(7.5,1.5)
(7.5,1.5)--(9,0)

(.8,1.15)--(.8,.8)
(.8,.8)--(1.1,.8)

(1.9,.7)--(2.25,.75)
(2.25,.75)--(2.23,1.1)

(3.8,1.15)--(3.8,.8)
(3.8,.8)--(4.1,.8)

(4.9,.7)--(5.25,.75)
(5.25,.75)--(5.23,1.1)

(6.8,1.15)--(6.8,.8)
(6.8,.8)--(7.1,.8)

(7.9,.7)--(8.25,.75)
(8.25,.75)--(8.23,1.1);

\filldraw[black]
(0,-2) circle [radius=.06] node [left] {7}
(3,-3.5) circle [radius=.06] node [above] {5}
(3,-2) circle [radius=.06] node [left] {8}
(6,-2) circle [radius=.06] node [right] {9}
(6,-3.5) circle [radius=.06] node [above] {6}
(9,-2) circle [radius=0.06] node [right] {10}
(4.5,-5) circle [radius=.06] node [below] {4}
(4.5,-6) circle [radius=0] node [below] {$G_2$};

\draw
(0,-2)--(3,-3.5)
(0,-2)--(4.5,-5)
(4.5,-5)--(9,-2)
(3,-3.5)--(6,-2)
(3,-2)--(6,-3.5)
(6,-3.5)--(9,-2)

(3.3,-4.2)--(3.4,-4.6)
(3.3,-4.2)--(3.7,-4.2)

(2,-3)--(2.1,-3.3)
(2,-3)--(2.3,-2.9)

(5,-3)--(5.1,-3.3)
(5,-3)--(5.3,-2.9)

(3.8,-2.9)--(4.15,-2.9)
(4.15,-2.9)--(4.1,-3.3)

(6.8,-2.9)--(7.15,-2.9)
(7.15,-2.9)--(7.1,-3.3)

(5.35,-4.2)--(5.65,-4.2)
(5.65,-4.2)--(5.6,-4.5);

\filldraw[black]
(12,0) circle [radius=.06] node [left] {7}
(13.5,1.5) circle [radius=.06] node [above] {1}
(15,0) circle [radius=.06] node [left] {8}
(16.5,1.5) circle [radius=.06] node [above] {2}
(18,0) circle [radius=.06] node [right] {9}
(19.5,1.5) circle [radius=.06] node [above] {3}
(21,0) circle [radius=0.06] node [right] {10};
\draw
(12,0)--(13.5,1.5)
(13.5,1.5)--(15,0)
(15,0)--(16.5,1.5)
(16.5,1.5)--(18,0)
(18,0)--(19.5,1.5)
(19.5,1.5)--(21,0);

\filldraw[black]
(15,-1.5) circle [radius=.06] node [above] {5}
(18,-1.5) circle [radius=.06] node [above] {6}
(16.5,-4) circle [radius=.06] node [below] {4}
(16.5,-6) circle [radius=0] node [below] {$G$};
\draw
(12,0)--(15,-1.5)
(12,0)--(16.5,-4)
(16.5,-4)--(21,0)
(15,-1.5)--(18,0)
(15,0)--(18,-1.5)
(18,-1.5)--(21,0)

(12.8,1.15)--(12.8,.8)
(12.8,.8)--(13.1,.8)

(13.9,.7)--(14.25,.75)
(14.25,.75)--(14.23,1.1)

(15.8,1.15)--(15.8,.8)
(15.8,.8)--(16.1,.8)

(16.9,.7)--(17.25,.75)
(17.25,.75)--(17.23,1.1)

(18.8,1.15)--(18.8,.8)
(18.8,.8)--(19.1,.8)

(19.9,.7)--(20.25,.75)
(20.25,.75)--(20.23,1.1)

(15.3,-2.95)--(15.4,-3.35)
(15.3,-2.95)--(15.7,-2.95)

(14,-1)--(14.1,-1.3)
(14,-1)--(14.3,-.9)

(17,-1)--(17.1,-1.3)
(17,-1)--(17.3,-.9)

(15.8,-.9)--(16.15,-.9)
(16.15,-.9)--(16.1,-1.3)

(18.8,-.9)--(19.15,-.9)
(19.15,-.9)--(19.1,-1.3)

(17.35,-3)--(17.65,-3)
(17.65,-3)--(17.6,-3.3);
\end{tikzpicture}
\caption{Gluing $G_1$ and $G_2$ at the sinks}\label{figure:gluing at sinks}
\end{figure}
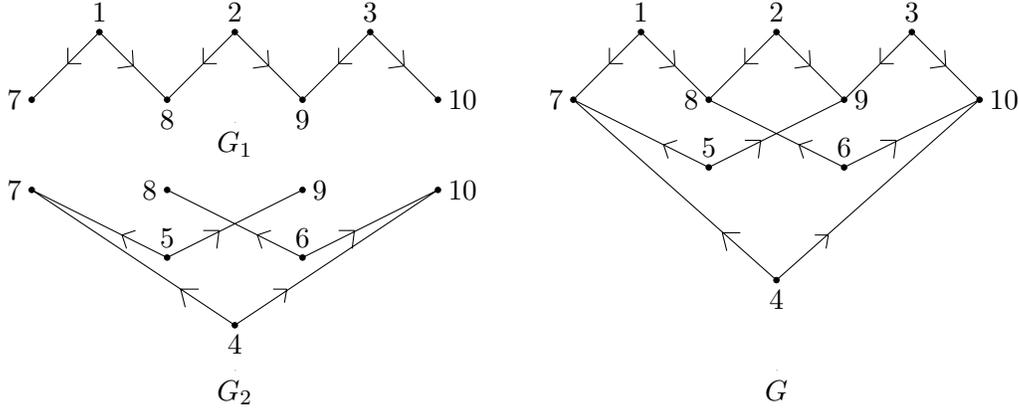

\begin{ex}\label{ex:gluing at sinks}
Let $G_1$ and $G_2$ be two DAGs as shown in Figure 
\ref{figure:gluing at sinks}. 
Here, the set of sinks in both $G_1$ and $G_2$ are $S_1=S_2=S=\{7,8,9,10\}$. 
We glue $G_1$ and $G_2$ at the sinks to form $G$. 
\end{ex}

\begin{thm}\label{thm: multiple sinks}
Let $G_1$ and $G_2$ be two DAGs.  Let $S$ be the set of common sinks in $G_1$ and $G_2$. 
Let $G$ be the DAG obtained after gluing $G_1$ and $G_2$ at the sinks. 
Suppose that for each pair of vertices $i,j \in S$,  either all treks
between $i$ and $j$ lie in $G_1$ or all treks between $i$ and $j$ lie in $G_2$. 
Then
\begin{eqnarray*}
I_G &=&  \, \, \, \,  \langle \text{ generators of }I_{G_1}\setminus \{\sigma_{ij} : i,j \in S \} \rangle  \\
&  &  + \langle \text{ generators of }I_{G_2}\setminus \{\sigma_{ij} : i,j \in S \} \rangle \\ 
&  &+\langle \sigma_{ij} : i \in V(G_1)\setminus S, j \in V(G_2)\setminus S \rangle  \\
&  &   + \langle \sigma_{ij} : i,j \in S \text{ such that there is no trek between } i \text{ and } j \rangle.
\end{eqnarray*}
\end{thm}

\begin{rmk} 
From the condition mentioned in the statement, we know that at 
least one of $\phi_{G_1}(\sigma_{ij})$ or $\phi_{G_2}(\sigma_{ij})$ 
is zero for all $i,j\in S, i\neq j$. 
So, ``$\langle \text{ generators of } I_{G_1}\setminus \{\sigma_{ij} : i,j \in S\} \rangle$", refers to forming a 
homogeneous generating set of $I_{G_1}$ that includes those indeterminates  in $\{\sigma_{ij}: i,j \in S\} $ 
which are mapped to zero under $\phi_{G_1}$ and then removing those indeterminates from the generating set.
Similarly, for 
``$\langle \text{ generators of } I_{G_2}\setminus \{\sigma_{ij} : i,j \in S\} \rangle$".
\end{rmk}

\begin{proof}
From the assumption that $S$ is a set of sinks of $G$, 
we know that there is no trek in $G$ between the vertices of $G_1\setminus S$ and $G_2\setminus S$. 
This implies that $\sigma_{ij}\in I_G$ for all $i\in V(G_1)\setminus S, j\in V(G_2)\setminus S$. 
Further, no two sinks $i,j$ in $S$ can have treks $i\rightleftharpoons j$ in both 
$G_1$ and $G_2$. So, the map $\phi_G$ can be written as 
\begin{eqnarray*}
\phi_G(\sigma_{ij})= \begin{cases} \phi_{G_1}(\sigma_{ij}) :  i \in V(G_1), j \in V(G_1)\setminus S \\
\phi_{G_2}(\sigma_{ij}) :  i \in V(G_2), j \in V(G_2)\setminus S  \\
\phi_{G_1}(\sigma_{ij})+\phi_{G_2}(\sigma_{ij}) : i,j \in S \\
\hspace{.4cm} 0 \hspace{.8cm}:  i \in V(G_1)\setminus S, j\in V(G_2)\setminus S .
\end{cases} 
\end{eqnarray*}
This allows us to partition the indeterminates $\sigma_{ij}$ into four sets $A_1, A_2, A_3, A_4$ where
\begin{eqnarray*}
A_1&=& \{\sigma_{ij} : i\in V(G_1), j\in V(G_1)\setminus S \text{ or } i,j\in S \text{ such that the treks } i\rightleftharpoons j \text{ lie in } G_1\} \\
A_2&=& \{\sigma_{ij} : i\in V(G_2), j\in V(G_2)\setminus S \text{ or } i,j\in S \text{ such that the treks } i\rightleftharpoons j \text{ lie in } G_2\} \\
A_3&=& \{\sigma_{ij} : i \in V(G_1)\setminus S, j \in V(G_2)\setminus S  \}  \\
A_4 & = &  \{\sigma_{ij} :  i,j \in S \text{ such that there is no trek between } i \text{ and } j \}.
\end{eqnarray*}
In these four sets, $\phi_G(\sigma_{ij})$ appear in disjoint sets of 
indeterminates and there can be no nontrivial relations involving two or more of 
these sets of indeterminates. So,
\begin{eqnarray*}
I_G = \ker \phi_G &=& \langle \text{ generators of }I_{G_1}\setminus \{\sigma_{ij} : i,j \in S \}
\rangle + \langle \text{ generators of }I_{G_2}\setminus \{\sigma_{ij} : i,j \in S \} 
\rangle \\ 
&&+\langle \sigma_{ij} : \sigma_{ij} \in A_3 \cup A_4 \rangle.  
\end{eqnarray*}
This completes the proof.
\end{proof}

\begin{ex}
Going back to Example \ref{ex:gluing at sinks}, we compute the vanishing ideals of the three DAGs $G_1$, $G_2$, and $G$. 
That gives us
\begin{eqnarray*}
I_{G_1}&=& \langle \sigma_{23},\sigma_{13},\sigma_{12},\sigma_{8,10},\sigma_{7,10},\sigma_{79},\sigma_{2,10},\sigma_{38},\sigma_{1,10},\sigma_{37},\sigma_{19},\sigma_{27},\sigma_{39}\sigma_{3,10}-\sigma_{33}\sigma_{9,10},\\
&&\sigma_{28}\sigma_{29}-\sigma_{22}\sigma_{89},\sigma_{17}\sigma_{18}-\sigma_{11}\sigma_{78} \rangle, \\
I_{G_2}&=& \langle \sigma_{9,10},\sigma_{89},\sigma_{78},\sigma_{69},\sigma_{5,10},\sigma_{67},\sigma_{49},\sigma_{58},\sigma_{48},\sigma_{56},\sigma_{46},\sigma_{45},\sigma_{68}\sigma_{6,10}-\sigma_{66}\sigma_{8,10},\\
&&\sigma_{47}\sigma_{4,10}-\sigma_{44}\sigma_{7,10},\sigma_{57}\sigma_{59}-\sigma_{55}\sigma_{79} \rangle, \\
I_{G}&=& \langle 
\sigma_{15},\sigma_{14},\sigma_{13},\sigma_{12},\sigma_{69},\sigma_{67},\sigma_{49},\sigma_{48},\sigma_{2,10},\sigma_{46},\sigma_{45},\sigma_{27},\sigma_{26},\sigma_{25},\sigma_{24},\sigma_{23},\sigma_{5,10},\\
&& \sigma_{58},\sigma_{56},\sigma_{38},\sigma_{1,10},\sigma_{19},\sigma_{37},\sigma_{36},\sigma_{35},\sigma_{16},\sigma_{34},
\sigma_{68}\sigma_{6,10}-\sigma_{66}\sigma_{8,10},\sigma_{28}\sigma_{29}-\sigma_{22}\sigma_{89},\\
&&\sigma_{47}\sigma_{4,10}-\sigma_{44}\sigma_{7,10},\sigma_{57}\sigma_{59}-\sigma_{55}\sigma_{79},\sigma_{39}\sigma_{3,10}-\sigma_{33}\sigma_{9,10},
\sigma_{17}\sigma_{18}-\sigma_{11}\sigma_{78}\rangle
\end{eqnarray*}

Observe that the indeterminates $\sigma_{27},\sigma_{79},\sigma_{37},\sigma_{7,10},\sigma_{12},
\sigma_{19},\sigma_{13},\sigma_{1,10},\sigma_{38},\sigma_{8,10},\sigma_{23},$ and  
$\sigma_{2,10}$ are mapped to zero by $\phi_{G_1}$ and the indeterminates
$ \sigma_{9,10},\sigma_{89},\sigma_{78},\sigma_{69},\sigma_{5,10},\sigma_{67},
\sigma_{49},\sigma_{58},\sigma_{48}, \sigma_{56}, 
\\
\sigma_{46}$ and  $\sigma_{45}$ 
are mapped to zero by $\phi_{G_2}$. 
Further, the treks $7\rightleftharpoons 8, 8\rightleftharpoons 9$ and $9\rightleftharpoons 10$ 
lie within $G_1$ whereas $7\rightleftharpoons 9, 7\rightleftharpoons 10$ and $8\rightleftharpoons 10$ 
lie within $G_2$. Also, no two sinks have treks between them in both $G_1$ and $G_2$.
Hence we are in a position where we can apply Theorem \ref{thm: multiple sinks}.

Analyzing the generating set of $I_G$, we see that the indeterminates 
$
\{\sigma_{13},\sigma_{12},\sigma_{2,10},\sigma_{27},\sigma_{23},\sigma_{38},$
\\
$\sigma_{1,10},\sigma_{19},\sigma_{37}\}$
and the binomials $\{\sigma_{28}\sigma_{29}-\sigma_{22}\sigma_{89},
\sigma_{39}\sigma_{3,10}-\sigma_{33}\sigma_{9,10},\sigma_{17}\sigma_{18}-\sigma_{11}\sigma_{78}\}$
in the generating set of $I_G$ are obtained from the generating set of $I_{G_1}$ 
after removing the indeterminates of the form $\{\sigma_{ij}: i,j\in S\}$.
Similarly, the indeterminates $\{\sigma_{69},\sigma_{67},\sigma_{49},\sigma_{48},\sigma_{46},
\sigma_{45},\sigma_{5,10},\sigma_{58},\sigma_{56}\}$ and the binomials $\{\sigma_{68}\sigma_{6,10}-\sigma_{66}\sigma_{8,10},
\sigma_{47}\sigma_{4,10}-\sigma_{44}\sigma_{7,10},\sigma_{57}\sigma_{59}-\sigma_{55}\sigma_{79}\}$
 are obtained from the generating set of $I_{G_2}$ after removing the indeterminates  $\{\sigma_{ij} : i,j\in S\}$.
The indeterminates $\{\sigma_{15},\sigma_{14},\sigma_{26},\sigma_{25},\sigma_{24},\sigma_{36},\sigma_{35},\sigma_{16},\sigma_{34}\}$ correspond to the third set of 
generators which are indeterminates of the form $\{\sigma_{ij}:i\in V(G_1)\setminus S, j\in V(G_2)\setminus S\}$.
In this example, there are no generators of the form 
\[
\langle \sigma_{ij} : i,j \in S \text{ such that there is no trek between } i \text{ and } j \rangle.
\]
\end{ex}

If we add the extra condition in Theorem \ref{thm: multiple sinks} 
that both $I_{G_1}$ and $I_{G_2}$ are toric, then we get the following result :

\begin{cor}\label{cor:gluing at sinks-toric}
Let $G_1$ and $G_2$ be two DAGs.  Let $S$ be the set of common sinks in $G_1$ and $G_2$. 
Let $G$ be the DAG obtained after gluing $G_1$ and $G_2$ at the sinks. 
Suppose that for each pair of vertices $i,j \in S$,  either all treks
between $i$ and $j$ lie in $G_1$ or all treks between $i$ and $j$ lie in $G_2$. 
\begin{enumerate}[i)]
    \item If $I_{G_1}$ and $I_{G_2}$ are toric, then $I_G$ is also toric.
    \item If $I_{G_1}=ST_{G_1}$ and $I_{G_2}=ST_{G_2}$, then $I_{G}$ is also equal to $ST_{G}$.
\end{enumerate}
\end{cor}
\begin{proof}
Part i) follows directly from Theorem \ref{thm: multiple sinks}, since the generating set will
be a union of a set of indeterminates and a collection of binomials.
For part ii), the shortest trek map $\psi_G$  has the same 
structure as $\phi_G$ as shown in the proof of Theorem \ref{thm: multiple sinks}. 
\end{proof}

We now look at a simple construction where instead of gluing two DAGs at the sinks, 
we add a new sink vertex to an existing DAG $G$. 
We show that the new DAG $G'$ has the same vanishing ideal as the existing one. 

\begin{thm}\label{thm:one sink}
Let $G$ be any arbitrary DAG. Construct a new DAG $G'$ from $G$, where we add another vertex $s$ 
and all edges $i \to s$ for $i \in V(G)$.
Then 
\[
I_{G'}   =  I_G \cdot \cc[\sigma_{ij} :  i,j \in V(G) \cup \{s\} ].
\]
\end{thm}
\begin{proof}
Let $G$ have $n$ vertices and $e$ edges. 
From the construction, we know that $G'$ has $n+1$ vertices and 
$e+n$ edges. Since the new vertex $s$ is a sink,  
none of the treks between any two  vertices 
$i,j \in V(G') \setminus \{s\}$ can pass through $s$. 
Further, as $s$ is connected to every vertex of $G$, 
the image of $\sigma_{is}$ has a monomial of the form $a_i\lambda_{is}$ for all $i\in G$. 
Thus, the map $\phi_{G'}$ can be written as
\begin{eqnarray*}
\phi_{G'}(\sigma_{ij})= \begin{cases} \phi_{G}(\sigma_{ij}) :  i,j \in V(G)\\
a_i\lambda_{is}+\text{ other terms } :  i \in V(G), j=s   \\
\hspace{.4cm} a_s \hspace{.5cm} : i=j=s.
\end{cases} 
\end{eqnarray*}
Since $\phi_G(\sigma_{ij}) =  \phi_{G'}(\sigma_{ij})$ for all $i,j \in V(G') \setminus \{s\}$, 
it is clear that $I_G\subseteq I_{G'}$. 

In order to show that $I_{G'}   =  I_G \cdot \cc[\sigma_{ij} :  i,j \in V(G) \cup \{s\} ]$, 
we look at the dimension of the two ideals. 
We know that the dimension of $I_G$ is $n+e$, 
whereas the dimension of $I_{G'}\subseteq \mathbb{C}[\sigma_{ij}: i,j \in V(G')]$ is $(n+1)+(e+n)=2n+e+1$. 
The only new indeterminates present in $\mathbb{C}[\sigma_{ij}: i,j \in V(G')]$ 
are the indeterminates of the form $\sigma_{is}:i\in V(G')$. 
So, the dimension of $I_G$ in $\mathbb{C}[\sigma_{ij}: i,j \in V(G')]$ is $n+e+(n+1)$, 
which equals the dimension of $I_{G'}$. But as $I_G\subseteq I_{G'}$
and both ideals are prime, we can conclude that $I_G=I_{G'}$. 
\end{proof}

Again, if we add the extra condition that $I_G$ is toric in Theorem \ref{thm:one sink}, 
then we get the following result :

\begin{cor}\label{cor:one sink-toric}
Let $G$ be any arbitrary DAG. Construct a new DAG $G'$ from $G$, where we add another vertex $s$ 
and all edges $i \to s$ for $i \in V(G)$.
\begin{enumerate}[i)]
    \item If $I_{G}$ is toric, then $I_{G'}$ is also toric.
    \item If $I_G=ST_G$, then $I_{G'}$ is also equal to $ST_{G'}$ and hence is toric.
\end{enumerate}
\end{cor}

\begin{proof}
For part i), since the two ideals have the same generating set, then they are both toric.

For part ii), using the same argument as in the Proof of Theorem \ref{thm:one sink}, the shortest trek map $\psi_{G'}$ can be written as 
\begin{eqnarray*}
\psi_{G'}(\sigma_{ij})= \begin{cases} \psi_{G}(\sigma_{ij}) :  i,j \in V(G)\\
\hspace{.2cm} a_i\lambda_{is} \hspace{.27cm}: i \in V(G), j=s   \\
\hspace{.4cm} a_s \hspace{.5cm} : i=j=s.
\end{cases} 
\end{eqnarray*}
So it is clear that $ST_G\subseteq ST_{G'}$. Now, the indeterminate 
$\lambda_{is}$ only appears in the image of $\sigma_{is}$ for all $i\in V(G)$. 
Similarly, the indeterminate $a_s$ only appears in the image of $\sigma_{ss}$. 
This implies that the indeterminates of the form $\sigma_{is}, i\in V(G')$ 
can not appear in any generators of $ST_{G'}$.  Thus  $ST_{G'} = ST_G \cdot \cc[\sigma_{ij} :  i,j \in V(G) \cup \{s\} ]$ as well, so $I_{G'} = ST_{G'}$.
\end{proof}

\begin{figure}
\begin{tikzpicture}[scale=1.2]
\filldraw[black]
(0,0) circle [radius=.04] node [below] {1}
(1.5,0) circle [radius=.04] node [below] {2}
(3.5,0) circle [radius=.04] node [below] {4}
(2.5,1) circle [radius=.04] node [above] {3}
(1.7,-.75) circle [radius=0] node [below] {$G$}; 
\draw
(0,0)--(3.5,0)
(1.5,0)--(2.5,1)
(2.5,1)--(3.5,0);

\draw
(.5,.15)--(.7,0)
(.5,-.15)--(.7,0)
(1.9,0.6)--(2.1,.6)
(2.1,.4)--(2.1,.6)
(2.3,.15)--(2.5,0)
(2.3,-.15)--(2.5,0)
(3,0.7)--(3,.5)
(3,.5)--(2.8,.5);

\filldraw[black]
(5,0) circle [radius=.04] node [below] {1}
(6.5,0) circle [radius=.04] node [above] {2}
(8.5,0) circle [radius=.04] node [below] {4}
(7.5,1) circle [radius=.04] node [above] {3}
(7,-1.5) circle [radius=.04] node [below] {5}
(6.5,-1.7) circle [radius=0] node [below] {$G'$}; 
\draw
(5,0)--(8.5,0)
(6.5,0)--(7.5,1)
(7.5,1)--(8.5,0)
(5,0)--(7,-1.5)
(6.5,0)--(7,-1.5)
(8.5,0)--(7,-1.5)
(7.5,1)--(7,-1.5);

\draw
(5.5,.15)--(5.7,0)
(5.5,-.15)--(5.7,0)
(6.9,0.6)--(7.1,.6)
(7.1,.4)--(7.1,.6)
(7.3,.15)--(7.5,0)
(7.3,-.15)--(7.5,0)
(8,0.7)--(8,.5)
(8,.5)--(7.8,.5)

(5.9,-.83)--(6.1,-.83)
(6.1,-.83)--(6.05,-.6)

(6.6,-.69)--(6.8,-.85)
(6.8,-.85)--(6.85,-.65)

(7.05,-.69)--(7.12,-.85)
(7.12,-.85)--(7.28,-.7)

(7.7,-.8)--(7.7,-.6)
(7.7,-.8)--(7.9,-.75);
\end{tikzpicture}
\caption{Introducing a new sink in $G$ to get $G'$}\label{figure:one sink}
\end{figure}
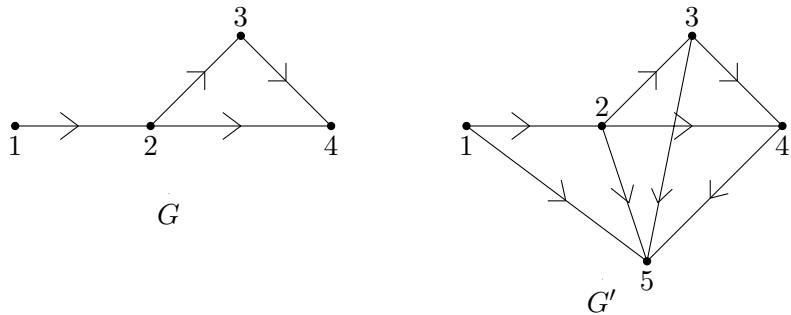

\begin{ex}
Let $G$ be a DAG with four vertices as shown in Figure \ref{figure:one sink}. 
From Example \ref{ex:3blocks}, we know that $I_G$ is a toric ideal. 
Now, we add another vertex $\{5\}$ to $G$ and connect all the existing vertices 
to $5$ by edges pointing towards $5$. Here $5$ is the sink in the new DAG $G'$. 
Computing the vanishing ideal of $G'$ gives us that $I_{G'}$ has the same generating set as $I_G$.
\end{ex}

To this point, 
we have described three ways to construct DAGs from smaller DAGs that preserve the toric
property:  safe gluing, gluing at sinks, and adding a new sink. 
We believe that these are the only possible operations that could be done to construct such DAGs. 
We know that the vanishing ideal of a complete DAG is zero and hence is toric. 
So starting with those examples as a base case, we can combine these three
operations to get many more examples of DAGs with toric vanishing ideals.
We explain this idea with an example.

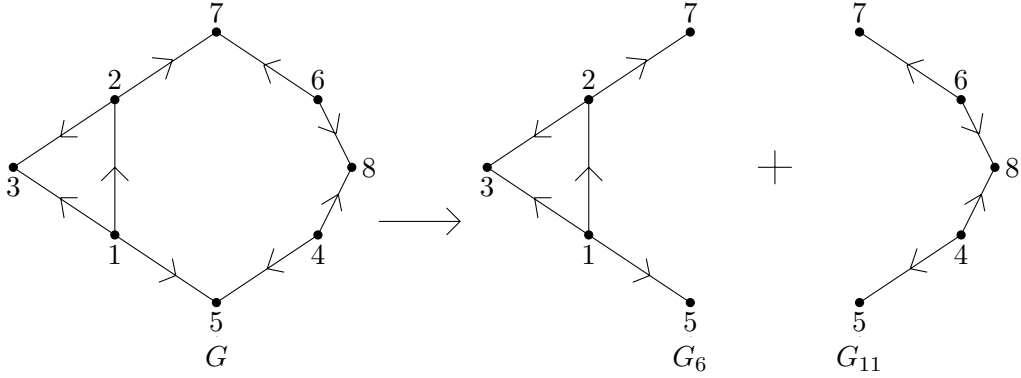
\begin{figure}
\begin{tikzpicture}[scale=0.9]
\filldraw[black]
(-1.5,1) circle [radius=.06] node [below] {3}
(0,0) circle [radius=.06] node [below] {1}
(0,2) circle [radius=.06] node [above] {2}
(1.5,-1) circle [radius=.06] node [below] {5}
(1.5,3) circle [radius=.06] node [above] {7}
(3,0) circle [radius=.06] node [below] {4}
(3,2) circle [radius=.06] node [above] {6}
(3.5,1) circle [radius=.06] node [right] {8}
(1.5,-1.5) circle [radius=0] node [below] {$G$};
\draw
(0,0)--(0,2)
(0,0)--(1.5,-1)
(-1.5,1)--(0,0)
(-1.5,1)--(0,2)
(0,2)--(1.5,3)
(1.5,-1)--(3,0)
(3,0)--(3.5,1)
(3,2)--(3.5,1)
(1.5,3)--(3,2)

(-.75,.25)--(-.8,.55)
(-.8,.55)--(-.55,.55)

(-.75,1.7)--(-.8,1.45)
(-.8,1.45)--(-.55,1.45)

(-.2,.8)--(0,1)
(.2,.8)--(0,1)

(.85,2.59)--(0.55,2.62)
(.85,2.59)--(0.75,2.3)

(2.45,2.59)--(2.2,2.51)
(2.3,2.24)--(2.2,2.51)

(.9,-.6)--(.65,-.7)
(.9,-.6)--(.85,-.35)

(2.25,-.5)--(2.35,-.2)
(2.25,-.5)--(2.5,-.55)

(3.05,.45)--(3.3,.6)
(3.3,.6)--(3.35,.35)

(3.32,1.8)--(3.27,1.5)
(3.27,1.5)--(3,1.6)

(3.9,0.2)--(5.1,0.2)
(5.1,0.2)--(4.9,0)
(5.1,0.2)--(4.9,0.4);

\filldraw[black]
(5.5,1) circle [radius=.06] node [below] {3}
(7,0) circle [radius=.06] node [below] {1}
(7,2) circle [radius=.06] node [above] {2}
(8.5,-1) circle [radius=.06] node [below] {5}
(8.5,3) circle [radius=.06] node [above] {7}
(8.5,-1.5) circle [radius=0] node [below] {$G_6$};

\draw
(7,0)--(7,2)
(7,0)--(8.5,-1)
(5.5,1)--(7,0)
(5.5,1)--(7,2)
(7,2)--(8.5,3)
(9.5,1)--(10,1)
(9.75,1.25)--(9.75,.75)

(6.25,.25)--(6.2,.55)
(6.2,.55)--(6.45,.55)

(6.25,1.7)--(6.2,1.45)
(6.2,1.45)--(6.45,1.45)

(7.85,2.59)--(7.55,2.62)
(7.85,2.59)--(7.75,2.3)

(7.9,-.6)--(7.65,-.7)
(7.9,-.6)--(7.85,-.35)

(6.8,.8)--(7,1)
(7.2,.8)--(7,1);

\filldraw[black]
(11,-1) circle [radius=.06] node [below] {5}
(11,3) circle [radius=.06] node [above] {7}
(12.5,0) circle [radius=.06] node [below] {4}
(12.5,2) circle [radius=.06] node [above] {6}
(13,1) circle [radius=.06] node [right] {8}
(11,-1.5) circle [radius=0] node [below] {$G_{11}$};

\draw
(11,-1)--(12.5,0)
(12.5,0)--(13,1)
(13,1)--(12.5,2)
(12.5,2)--(11,3)

(11.75,-.5)--(11.85,-.2)
(11.75,-.5)--(12,-.55)

(11.95,2.59)--(11.7,2.51)
(11.8,2.24)--(11.7,2.51)

(12.55,.45)--(12.8,.6)
(12.8,.6)--(12.85,.35)

(12.82,1.8)--(12.77,1.5)
(12.77,1.5)--(12.5,1.6);

\end{tikzpicture}
\caption{Constructing $G$ as a combination of safe gluing, gluing at sinks and adding a new sink to complete DAGs}\label{figure:combination}
\end{figure}

\begin{ex}
Let $G$ be the DAG as shown in Figure \ref{figure:combination}. Computing the vanishing ideal gives us that 
\begin{eqnarray*}
I_G &=& \langle
\sigma_{5,6}, \sigma_{4,7}, \sigma_{4,6}, \sigma_{3,8}, \sigma_{3,6}, \sigma_{3,4}, \sigma_{2,8}, \sigma_{2,6}, \sigma_{2,4}, \sigma_{1,8}, \sigma_{1,6}, \sigma_{1,4},
\sigma_{6,7}\sigma_{6,8}-\sigma_{6,6}\sigma_{7,8}, \\
&& \sigma_{4,5}\sigma_{4,8}-\sigma_{4,4}\sigma_{5,8}, \sigma_{2,5}\sigma_{3,7}-\sigma_{2,3}\sigma_{5,7}, 
\sigma_{1,7}\sigma_{3,5}-\sigma_{1,3}\sigma_{5,7}, \sigma_{2,5}\sigma_{2,7}-\sigma_{2,2}\sigma_{5,7}, \\
&& \sigma_{2,3}\sigma_{2,7}-\sigma_{2,2}\sigma_{3,7}, \sigma_{1,7}\sigma_{2,5}-\sigma_{1,2}\sigma_{5,7}, 
\sigma_{1,3}\sigma_{2,5}-\sigma_{1,2}\sigma_{3,5}, \sigma_{1,7}\sigma_{2,3}-\sigma_{1,2}\sigma_{3,7}, \\
&&\sigma_{1,7}\sigma_{2,2}-\sigma_{1,2}\sigma_{2,7}, \sigma_{1,5}\sigma_{1,7}-\sigma_{1,1}\sigma_{5,7}, 
\sigma_{1,3}\sigma_{1,5}-\sigma_{1,1}\sigma_{3,5}, \sigma_{1,2}\sigma_{1,5}-\sigma_{1,1}\sigma_{2,5}
\rangle.
\end{eqnarray*}
Now, we show that $G$ can be obtained as a combination of safe gluing, 
gluing at sinks, and adding a new sink starting from complete DAGs. 
Let $G_1$ be the DAG with vertices $\{1,2\}$. 
Then the vertex $3$ can be considered as adding a new sink to 
$G_1$ to form $G_2$. So, $G_2$ is the DAG with vertices $\{1,2,3\}$ 
and $I_{G_2}$ is toric. 

Let $G_3$ be the complete DAG with vertices $\{2,7\}$. Then we can make a safe gluing of 
$G_2$ with $G_3$ to get $G_4$ as $2$ is a choke point between $\{1,2,3\}$ and $\{2,7\}$. 
Similarly, if $G_5$ is the complete DAG with vertices $\{1,5\}$, then we can make another 
safe gluing of $G_4$ with $G_5$ to form $G_6$. Observe that $G_6$ has three sinks, 
which are $3,7$, and $5$. 

Let $G_7, G_8, G_9$ and $G_{10}$ be the complete DAGs with vertices 
$\{6,7\},\{6,8\},\{4,8\}$ and $\{4,5\}$ respectively. Then we can perform 
multiple safe gluings of these four DAGs to get $G_{11}$ with vertices 
$\{4,5,6,7,8\}$. It can be seen that $5$ and $7$ are the two sinks in 
$G_{11}$. So, finally we can glue $G_6$ and $G_{11}$ at the set of common sinks, 
i.e., $5$ and $7$. As there exist only one trek between $5$ and $7$ and that lies in 
$G_6$, we can conclude that the final DAG $G$ obtained after gluing $G_6$ and 
$G_{11}$ at the sinks must have a toric vanishing ideal.
\end{ex}


\section{The shortest trek ideal}\label{sec:shortesttrek}

The shortest trek ideal $ST_G$ appears to play an important role in the 
problem of classifying those DAGs whose vanishing ideal is toric.
For this reason, we focus on purely combinatorial properties of this
ideal in this section.  In particular, we prove our main result,
Theorem \ref{thm:ker(psi_G)=CI_G},  
that if  $ST_{G_1}$ equals $CI_{G_1}$ and 
$ST_{G_2}$ equals  $CI_{G_2}$, then $ST_{G}$ equals $CI_{G}$ where $G$ is a safe gluing of
$G_1$ and $G_2$.  This result provides further evidence for 
Conjecture \ref{conj:safe gluing conjecture}.

We begin with exploring the  structure of the shortest trek map. 

\begin{prop}\label{prop:dim of STG}
Let $G$ be a DAG such that the shortest trek ideal $ST_G$ exists.
Then the dimension of $ST_G$ is $n + e$, the number of vertices plus the number of edges.
\end{prop}

\begin{proof}
The number of parameters in the ring $\cc[a,\lambda]$ is $n + e$, so $n + e$ is an upper bound
on the dimension.  On the other hand, for each $i$, $\psi_G(\sigma_{ii}) = a_i$ and
for edge edge $i \to j$,  $\psi_G(\sigma_{ij}) =  a_i  \lambda_{ij}$.
This collection of expressions 
\[
\{a_i : i \in V(G) \}  \cup  \{ a_i \lambda_{ij} :  i \to j \in E(G) \}
\]
is algebraically independent, and has cardinality $n+e$ which gives a lower bound for
the dimension of $ST_G$.  
\end{proof}

As $\psi_G$ is a monomial map, there is a corresponding matrix $M$, whose columns are the
exponent vectors in the monomials $\psi_G(\sigma_{ij})$.   
So $ST_G$ is the toric ideal of the matrix $M$ as
\[
\psi_G(\sigma^u)=t^{M u},
\]
where $\sigma=(\sigma_{11},\sigma_{12},\ldots,\sigma_{nn})$ and 
$t=(a_1,a_2,\ldots,a_n,\lambda_{12},\ldots, \lambda_{n-1n})$.
This matrix will be useful in proving some properties of the ideal $ST_G$.

To prove results about the generating sets of toric ideals, it
is useful to consider the notion of a fiber graph.
For any vector $b\in \mathbb{N}^{(n+e)}$, 
the \emph{fiber} of $M$ over $b$ is defined as 
\[
M^{-1}(b)=\{u\in \mathbb{N}^{(n^2+n)/2}:M u=b\}.
\]
As the columns of $M$ are non-zero and non-negative, 
$M^{-1}(b)$ is always finite for any $b\in \mathbb{N}^{(n+e)}$. 
Let $\mathcal{F}$ be any finite subset of $\ker_\zz(M)$. 
The \emph{fiber graph}  $M^{-1}(b)_\mathcal{F}$ is defined as follows:
\begin{enumerate}[i)]
\item The nodes of this graph are the elements of $M^{-1}(b)$. 
\item Two nodes $u$ and $u'$ are connected by an edge if $u-u'\in \mathcal{F}$
 or $u'-u\in \mathcal{F}$. 
\end{enumerate}

The fundamental theorem of Markov bases connects the generating 
sets of toric ideals to connectivity properties of the fiber graphs. 
We state this explicitly in the case of the fiber graphs for the shortest trek maps.

\begin{thm}\cite[Thm 5.3]{GB and CP}\label{thm:connected-generators}
Let $\mathcal{F}\subset \ker_\zz(M)$. 
The graphs $M^{-1}(b)_\mathcal{F}$ are connected for all $b$ such that $M^{-1}(b)$ is nonempty,
if and only if the set 
$\{\sigma^{v^+}-\sigma^{v^-}:v\in \mathcal{F}\}$ generates the toric ideal $ST_G$.
\end{thm}

Now we relate the toric ideal $ST_G$ to some other familiar toric ideals
that are studied in the combinatorial algebra literature.  These results
will be useful for proving results on the generators of $ST_G$.
\begin{defn}
We define a map called the 
\textit{end point map $\eta_G$} as follows:
\begin{eqnarray*}
\eta_G:\mathbb{C}[\sigma_{ij} : 1\leq i\leq j\leq n] &\rightarrow& \mathbb{C}[d_1,\ldots,d_n] \\
\sigma_{ij} &\mapsto& 
\begin{cases}  
d_id_j  &  \mbox{ if there is a trek from $i$ to $j$} \\
0   &   \mbox{ otherwise}
\end{cases}
\end{eqnarray*}
As $\eta_G$ is also a monomial map, $\ker(\eta_G)$ is a toric ideal. 
\end{defn}

\begin{lemma}\label{lemma:end point map}
For any given DAG $G$ where the shortest trek map $\psi_G$ is well defined,
\[
ST_G\subseteq \ker (\eta_G).
\]
\end{lemma}

\begin{proof}
Let $M$ and $N$ be the matrices corresponding to the maps 
$\psi_G$ and $\eta_G$ respectively.
Note that we can ignore all pairs $i,j$ where there is no trek between $i$ and $j$,
as these $\sigma_{ij}$ map to zero under both the simple trek rule and the shortest trek 
map.
It is enough to show that the row space of $N$ is contained in the row space of $M$. 
We construct a matrix $M_1$ as follows:
\begin{enumerate}[i)]
    \item $M_1$ is an $n \times (n+|E|)$ matrix, 
    where the rows correspond to the vertices of $G$ (i.e, the indeterminates $d_i$) 
    and the columns correspond to the vertices and edges of $G$ (i.e, the indeterminates $a_i$ and $\lambda_{ij}$).
    \item For every vertex indeterminate $a_i$, the corresponding column is $2e_i$ and for every edge indeterminate $\lambda_{ij}$, the corresponding column is $-e_i+e_j$, where $e_i$ is the $i$th standard 
    unit vector.
\end{enumerate}
Now, let $\psi_G (\sigma_{ij})= a_k\lambda_{ki_1}\lambda_{i_1i_2}\cdots\lambda_{i_si}\lambda_{kj_1}\lambda_{j_1j_2}\cdots\lambda_{j_tj},$
where $k$ is the topmost vertex within the shortest trek $i\leftrightarrow j$. As $\psi_G(\sigma_{ij})=t^{Mu_{ij}}$ where $\sigma^{u_{ij}}=\sigma_{ij}$, we have 
\begin{eqnarray*}
M_1Mu_{ij}&=& 2e_k-e_k+e_{i_1}-e_{i_1}+e_{i_2}-\cdots-e_{i_s}+e_i-e_k+e_{j_1}-e_{j_1}+e_{j_2}-\cdots-e_{j_t}+e_j \\
&=&e_i+e_j \\
&=& Nu_{ij},
\end{eqnarray*}
for all $\sigma_{ij}, 1\leq i\leq j \leq n$. This implies that $N=M_1M$, which shows that $N$ is contained in the row space of $M$ and thus completes the proof.
\end{proof}

A consequence of Lemma \ref{lemma:end point map} is that the ideal $ST_G$ is homogeneous
with respect to the grading by indices.  So, if $\sigma^u - \sigma^v$ is in $ST_G$,
and all indeterminates involved correspond to actual treks, 
then, for each $i$, the index $i$ appears the same number of times in both $\sigma^u$ and
$\sigma^v$.  For example, it is not possible that $\sigma_{11}\sigma_{23} - \sigma_{13}\sigma_{24}$
is in any shortest trek ideal (unless some of these indeterminates correspond to pairs of vertices that
are not connected by treks).  

\begin{rmk}
Since the $\sigma_{ij}$ corresponding to pairs of vertices $i$ and $j$ with no trek
between them always appear as generators in the ideal $ST_G$, we need a way to ignore those
terms when speaking about binomials in $ST_G$.  Henceforth, when we speak of a binomial
$\sigma^u - \sigma^v$  in $ST_G$, we assume that all indeterminates appearing in this binomial
actually correspond to treks in $G$.
\end{rmk}

For a DAG $G$ if we want to show that $ST_G$ equals $CI_G$, 
it is enough to show that the set of $2\times 2$ minors of $\Sigma_{A\cup C,B\cup C}$ 
for all possible $d$-separations of $G$ form a generating set for $ST_G$. 
By using Theorem \ref{thm:connected-generators} this is equivalent to 
show that the graphs $M^{-1}(b)_\mathcal{F}$ is connected for all 
$b$, where $\mathcal{F}$ is the set of all $2\times 2$ minors of 
$\Sigma_{A\cup C,B\cup C}$  in the vector form, for all possible $d$-separations of $G$. 
Now, for a fixed $b$, let $u,v\in M^{-1}(b)_\mathcal{F}$. 
This implies that both $M u $ and $M v$ are equal to $b$,
which gives us $\psi_G(\sigma^u-\sigma^v)=0$. Therefore, 
it is enough to show that for any $f=\sigma^u-\sigma^v\in ST_G$, 
$\sigma^u$ and $\sigma^v$ are connected by the moves in $\cf$.

Now, for a DAG $G$ with $n$ vertices, let $u\in \mathbb{N}^{(n^2+n)/2}$ be a node in the graph of $M^{-1}(b)_\mathcal{F}$.  We in turn, represent this $u$, or equivalently the monomial $\sigma^u$, as a multi-digraph in the following way:
For each factor $\sigma_{ij}$ of $\sigma^u$ we draw all edges in the shortest trek
$i\leftrightarrow j$ along $G$ with highlighting the top vertices.  
For each $\sigma_{ii}$ we highlight that it is a top vertex. 
  
Let $\deg_i(\sigma^u)$ denote the \emph{degree} of a vertex $i$ in 
$\sigma^u$ which is defined to be the number of end points of paths in $\sigma^u$.
We count the loops corresponding to $\sigma_{ii}$ as having two endpoints at $i$. 
If $f=\sigma^u-\sigma^v$ is a homogeneous binomial in $ST_G$, then 
$\psi_G(\sigma^u)=\psi_G(\sigma^v)$ if and only if the following conditions are
satisfied:

\begin{enumerate}[i)]
\item The graphs of $\sigma^u$ and $\sigma^v$ both have the 
same number of treks (as $f$ is homogeneous), 
\item The graphs of $\sigma^u$ and $\sigma^v$ have the same number 
of edges between any two adjacent vertices $i$ and $j$ 
(as the exponent of $\lambda_{ij}$ in $\psi_G(\sigma^u)$ gives the number 
of edges between $i$ and $j$ in the graph of $\sigma^u$), 
\item  The multiset of top vertices in both graphs is the same.
\item   The degree of any vertex in both the graphs is the same 
(as $ST_G$ is contained in the kernel of $\eta_G$ by Lemma \ref{lemma:end point map} ).
\end{enumerate}

\begin{ex}
Let $G$ be the DAG as shown in Figure \ref{figure:ST_G connectivity}. From Example \ref{ex:shortest trek} (ii), we know that
\[
I_G=ST_G= CI_G= \langle \sigma_{12}\sigma_{23}-\sigma_{13}\sigma_{22},\sigma_{12}\sigma_{24}-\sigma_{14}\sigma_{22},\sigma_{13}\sigma_{24}-\sigma_{14}\sigma_{23} \rangle.  
\]
So, by Theorem \ref{thm:connected-generators}, we know that $\sigma^u$ and $\sigma^v$ are connected by the moves in $\mathcal{F}$ for any $\sigma^u-\sigma^v \in ST_G$, where $\mathcal{F}$ is the set of $2\times 2$ minors of $\Sigma_{A\cup C,B \cup C}$ in the vector form for all possible $d$-separations of $G$. Now, let 
\[
f=\sigma^u-\sigma^v=\sigma_{12}^2\sigma_{24}\sigma_{23}-\sigma_{22}^2\sigma_{13}\sigma_{14} \in ST_G.
\]
The multi-digraphs of $\sigma^u$ and $\sigma^v$ are as shown in Figure \ref{figure:fiber graphs of sigma}. Observe that the graphs of both $\sigma^u$ and $\sigma^v$ have four treks each. The number of edges $1\rightarrow 2$, $2\rightarrow 3$ and $2\rightarrow 4$ are 2,1 and 1 respectively in both the graphs. Further, the degree of each vertex $\{1\},\{2\},\{3\}$ and $\{4\}$ are also 2,4,1 and 1 respectively in both the graphs.

We can reach from $\sigma^u$ to $\sigma^v$ by first applying the move which takes $\sigma_{12}\sigma_{24}$ to $\sigma_{22}\sigma_{14}$ and then applying the move which takes $\sigma_{12}\sigma_{23}$ to $\sigma_{13}\sigma_{22}$.

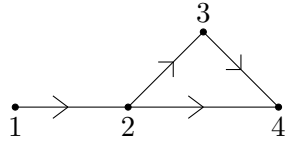
\begin{figure}
\begin{tikzpicture}
\filldraw[black]
(7.5,-1) circle [radius=.04] node [below] {1}
(9,-1) circle [radius=.04] node [below] {2}
(11,-1) circle [radius=.04] node [below] {4}
(10,0) circle [radius=.04] node [above] {3}; 
\draw
(7.5,-1)--++(3.5,0)
(9,-1)--++(1,1)
(10,0)--++(1,-1);

\draw
(8,-.85)--(8.2,-1)
(8,-1.15)--(8.2,-1)
(9.4,-.4)--(9.6,-.4)
(9.6,-.6)--(9.6,-.4)
(9.8,-.85)--(10,-1)
(9.8,-1.15)--(10,-1)
(10.5,-.3)--(10.5,-.5)
(10.3,-.5)--(10.5,-.5);

\end{tikzpicture}
\caption{A DAG $G$ where $I_G=ST_G$\label{figure:ST_G connectivity}}
\end{figure}

\begin{figure}
\begin{tikzpicture}
\filldraw[black]
(0.5,-1) circle [radius=.05] 
(0.5,-1.2) circle [radius=.05]
(1.9,-1.2) circle [radius=.02]
(0.5,-1.2) circle [radius=0] node [below] {1}
(1.9,-1) circle [radius=.02] 
(2.2,-1.1) circle [radius=.05]
(2.1,-.8) circle [radius=.05]
(2,-1.2) circle [radius=0] node [below] {2}
(4,-1.1) circle[radius=.02]
(4,-1.2) circle [radius=0] node [below] {4}
(3,0) circle [radius=.02] node [above] {3}

(2,-2) circle [radius=0] node[below] {$\sigma^u$}; 

\draw
(0.5,-1)--(1.9,-1)
(0.5,-1.2)--(1.9,-1.2)
(2.2,-1.1)--(4,-1.1)
(2.1,-.8)--(3,0);

\draw
(1,-1.2)--(0.9,-1.1)
(1,-1.2)--(0.9,-1.3)
(1.2,-1)--(1.1,-0.9)
(1.2,-1)--(1.1,-1.1)
(2.55,-0.4)--(2.4,-0.4)
(2.55,-0.4)--(2.55,-0.55)
(3.2,-1.1)--(3.1,-1)
(3.2,-1.1)--(3.1,-1.2);

\filldraw[black]
(7.5,-1) circle [radius=.05] 
(7.5,-1.2) circle [radius=.05] node [below] {1}
(8.8,-1.4) circle [radius=.05] 
(9.2,-1.4) circle [radius=.05]
(9.1,-.9) circle [radius=0] node [right] {2}
(11,-1.2) circle [radius=.02] node [below] {4}
(10,0) circle [radius=.02] node [above] {3}
(9,-2) circle [radius=0] node[below] {$\sigma^v$}; 

\draw
(7.5,-1)--(8.9,-1)
(7.5,-1.2)--(11,-1.2)
(8.9,-1)--(10,0);

\draw
(8,-1.2)--(7.9,-1.1)
(8,-1.2)--(7.9,-1.3)
(8.2,-1)--(8.1,-0.9)
(8.2,-1)--(8.1,-1.1)
(9.55,-0.4)--(9.4,-0.4)
(9.55,-0.4)--(9.55,-0.55)
(10.2,-1.2)--(10.1,-1.1)
(10.2,-1.2)--(10.1,-1.3)
(8.75,-1.78)--(8.65,-1.68)
(8.75,-1.78)--(8.65,-1.88)
(9.35,-1.78)--(9.25,-1.68)
(9.35,-1.78)--(9.25,-1.88);

\draw
(8.7,-1.58) circle [radius=.2]
(9.3,-1.58) circle [radius=.2];

\filldraw[white]
(8.665,-1.4) circle [radius=0.065]
(9.322,-1.4) circle [radius=0.065];

\filldraw[black]
(8.6,-1.4) circle [radius=.02]
(9.4,-1.4) circle [radius=.02];

\end{tikzpicture}
\caption{The multi-digraphs of $\sigma^u$ and $\sigma^v$\label{figure:fiber graphs of sigma}}
\end{figure}
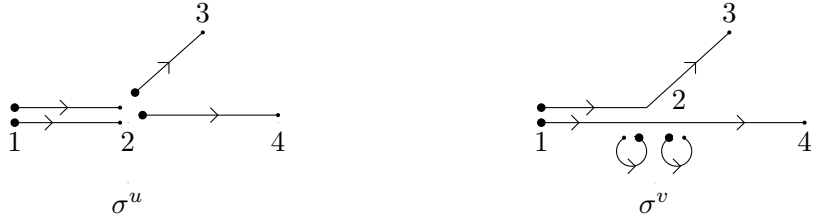
\end{ex}

\begin{lemma}\label{lemma:invalid move}
Let $G$ be a safe gluing of $G_1$ and $G_2$ such that $ST_{G_1}=CI_{G_1}$ and $ST_{G_2}=CI_{G_2}$. Then the set of all the $2\times 2$ minors of $\Sigma_{A\cup c,B\cup c}$ lie in $ST_G$, where $A=V(G_1)\setminus C$ and $B=V(G_2)\setminus C$.
\end{lemma}
\begin{proof}
Let $M$ be the set of all the $2\times 2 \text{ minors of } \Sigma_{A\cup c, B\cup c} $. These minors correspond to the separation criterion that $\{c\}$ $d$-separates $A$  from $B$. Every element in $M$ is of the form $\sigma_{ij}\sigma_{kl}-\sigma_{il}\sigma_{kj}$, where $i,l \in A\cup c$ and $j,k \in B\cup c$. Now, if all the four shortest treks $i \leftrightarrow j, k \leftrightarrow l, i \leftrightarrow l$ and $k \leftrightarrow j$ contain $c$, then each of these four treks can be decomposed as
\begin{eqnarray*}
i\leftrightarrow j &=& i\leftrightarrow c \cup c\leftrightarrow j, \\ 
k \leftrightarrow l &=& k \leftrightarrow c \cup c\leftrightarrow l, \\
i \leftrightarrow l &=& i\leftrightarrow c \cup c \leftrightarrow l , \\
k \leftrightarrow j &=& k\leftrightarrow c \cup c\leftrightarrow j. 
\end{eqnarray*}
From this decomposition, it is clear that $\sigma_{ij}\sigma_{kl}$ covers the same set of edges as $\sigma_{il}\sigma_{kj}$ and hence $\sigma_{ij}\sigma_{kl}-\sigma_{il}\sigma_{kj} \in ST_G$.

If one of these four shortest treks does not pass through $c$, then we cannot have a decomposition as above and hence cannot imply that the binomial lies in $ST_G$. Thus, we need to show that such a binomial does not appear in $M$.

Let  $f=\sigma_{ij}\sigma_{kl}-\sigma_{il}\sigma_{kj}$, where $i,l\in A \cup c, k,j\in B \cup c$ and the shortest trek $i \leftrightarrow l$ does not pass through $c$. Then the two monomials $\sigma_{ij}\sigma_{kl}$ and $\sigma_{il}\sigma_{kj}$ do not preserve the number of edges between adjacent vertices. To illustrate this, let us consider the vertex $c'$ which is adjacent to $c$ and lies in $i\leftrightarrow c$ (Fig \ref{figure:invalid move} (i)). (The shortest trek $i \leftrightarrow l$ here passes through the dashed line.) We observe that the multi-digraph of $\sigma_{ij}\sigma_{kl}$ contains the edge $c'\rightarrow c$ but the multi-digraph of $\sigma_{il}\sigma_{kj}$ does not contain $c'\rightarrow c$ as $i\leftrightarrow l$ does not pass through $c$. So, we need to show that $f\notin M$.

Now, all the possible options for DAGs which could fit in the above situation can be classified into two categories. This categorization is independent of the directions in $c\leftrightarrow k$ and $c\leftrightarrow j$ and is as follows :

\textbf{Case I :} The path between $i$ and $j$ containing $c$ has a collider at $c$ :

We illustrate this case in Fig \ref{figure:invalid move}, (i). Here, the shortest trek $i \leftrightarrow l$ is the trek which passes through the dashed line. Observe that $c$ can $d$-separate $i$ from $j$ and $k$ from $l$ 
but it cannot $d$-separate $i$ from $l$.
Similarly, any vertex which lies in $c_1\leftrightarrow c_2$ can 
$d$-separate $i$ from $l$ but they cannot $d$-separate $i$ from 
$j$ and $k$ from $l$ simultaneously. So, there does not exist any $2\times 2$ minor in $M$ where  $\sigma_{il}$ and $\sigma_{ij}$ or $\sigma_{kl}$ can occur together.

\textbf{Case II:} The path between $i$ and $j$ containing $c$ does not have a collider at $c$ :

In this case (Fig \ref{figure:invalid move}, (ii)), we see that $c$ alone cannot $d$-separate $i$ and $l$. 
So, we cannot have a binomial in $M$ with $\sigma_{il}$ as one of its terms.

Hence we can conclude that every element in $M$ lies in $ST_G$.
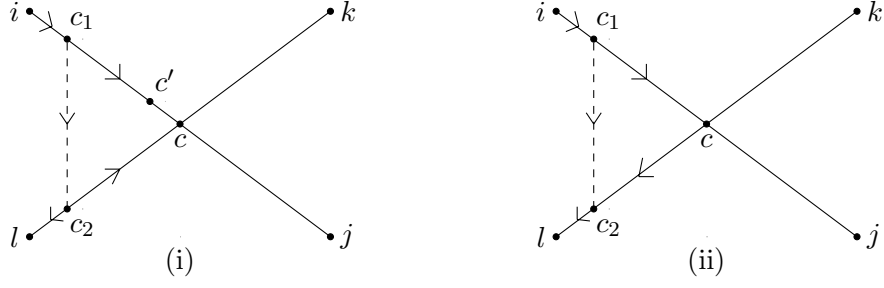
\begin{figure}
\begin{tikzpicture}
\filldraw[black]
(-2,0) circle [radius=0]
(0,0) circle [radius=.04] node[left] {$l$}    
(2,1.5) circle [radius=.04] node[below] {$c$}
(2,0) circle [radius=0] node[below] {(i)}
(4,0) circle [radius=.04] node[right] {$j$}
(0,3) circle [radius=.04] node[left] {$i$} 
(4,3) circle [radius=.04] node[right] {$k$}
(.5,.37) circle [radius=.04]  
(.7,.37) circle [radius=.0] 
node[below] {$c_2$}
(.5,2.63) circle [radius=.04] 
(.7,2.63) circle [radius=.0] node [above] {$c_1$}
(1.6,1.8) circle [radius=.04] 
(1.8,1.8) circle [radius=0] node [above] {$c'$};
\draw
(0,0)--(4,3)
(0,3)--(4,0)
(1.2,.9)--(1,.9)
(1.2,.9)--(1.1,.7)
(1.2,2.1)--(1.2,2.3)
(1.2,2.1)--(1,2.1)
(.5,1.5)--(.4,1.65)
(.5,1.5)--(.6,1.65)
(.3,2.8)--(.25,3)
(.3,2.8)--(.13,2.73)
(.28,.23)--(.28,.4)
(.28,.23)--(.45,.21);

\draw[dashed] (.5,.37)--(.5,2.63);

\filldraw[black]
(7,0) circle [radius=.04] node[left] {$l$}     
(9,1.5) circle [radius=.04] node[below] {$c$}
(9,0) circle [radius=0] node[below] {(ii)}
(11,0) circle [radius=.04] node[right] {$j$}
(7,3) circle [radius=.04] node[left] {$i$} 
(11,3) circle [radius=.04] node[right] {$k$}
(7.5,.37) circle [radius=.04] 
(7.7,.37) circle [radius=.0] 
node[below] {$c_2$}
(7.5,2.63) circle [radius=.04] 
(7.7,2.63) circle [radius=.0] node [above] {$c_1$}
;
\draw
(7,0)--(11,3)
(7,3)--(11,0)

(8.1,.8)--(8.3,.8)
(8.1,.8)--(8.15,1.05)
(8.2,2.1)--(8.2,2.3)
(8.2,2.1)--(8,2.1)
(7.5,1.5)--(7.4,1.65)
(7.5,1.5)--(7.6,1.65)

(7.3,2.8)--(7.25,3)
(7.3,2.8)--(7.13,2.73)

(7.28,.23)--(7.28,.4)
(7.28,.23)--(7.45,.21);
\draw[dashed] (7.5,.37)--(7.5,2.63);
\end{tikzpicture}
\caption{Two possible types of cases where an invalid move is possible \label{figure:invalid move}}
\end{figure}
\end{proof}

Suppose that $G$ can be written as a safe gluing of $G_1$ and $G_2$ at an $N$-clique. 
We define a map $\rho_{G_1}:  V(G) \rightarrow V(G_1)$ as follows:
\[ \rho_{G_1}(i) = \begin{cases} 
          i & i \in V(G_1) \\
              c & i \in V(G_2)\setminus C
       \end{cases}
    \]
where $C$ is the clique at which $G_1$ and $G_2$ are glued 
and $c$ is the special vertex in $C$.     
We can lift $\rho_{G_1}$  as a map between from 
$\cc[\Sigma]$ to itself by the rule 
$\rho_{G_1}(\sigma_{ij})=\sigma_{\rho_{G_1}(i)\rho_{G_1}(j)}$.

For a vector $u \in \nn^{n(n+1)/2}$,
let $u_{G_1}$ be the vector that extracts all the 
coordinates that correspond to the shortest treks that do not lie within $G_2$.  That is, 

\[
u_{G_1}(ij)= \begin{cases}
 0 & i,j \in G_2\setminus C \\
 u(ij) & \mbox{ otherwise. }

 \end{cases}
 \]
Then we have the following result.

\begin{prop} \label{prop:ST_G-c}
Let $G$ be a safe gluing of $G_1$ and $G_2$, with the map $\rho_{G_1}$ defined as above.
Suppose that $\sigma^u - \sigma^v \in ST_G$ 
and this binomial only involves $\sigma_{ij}$ indeterminates corresponding to treks.
Then
\[ 
\psi_{G_1}(\rho_{G_1}(\sigma^{u_{G_1}}))-\psi_{G_1}(\rho_{G_1}(\sigma^{v_{G_1}}))=0. 
\]
\end{prop}

Note that we use the notation $\psi_{G_1}$ to denote the shortest trek map associated to the graph $G_1$. However, the map
$\psi_G$ can also be used since that will give the same result. 

\begin{proof}
We have
\[ \rho_{G_1}(\sigma_{ij})=\begin{cases}
\sigma_{ij} & i,j \in V(G_1)\\
\sigma_{ic} & i \in V(G_1) \text{ and }  j\in V(G_2)\setminus C
\end{cases}
\]

We know that $\sigma^u$ and $\sigma^v$ have the same number of treks. 
Also, the degree of each vertex and the number of edges between 
any two adjacent vertices is the same. 
Moreover, the power of each $a_i$ (which corresponds to the source of every trek) is also the same.
So, it is enough to show that $\rho_{G_1}(\sigma^{u_{G_1}})$ and 
$\rho_{G_1}(\sigma^{v_{G_1}})$ have the same number of treks (which corresponds to the sum of all the powers of $a_i,i\in V(G_1)$ in the image) and the
number of edges between any two adjacent vertices (which we refer to as the degree of the edge) is also the same.

From the vector $u_{G_1}$ and the map $\rho_{G_1}$, we see that the treks in $\sigma^u$ of the form $i\leftrightarrow j$ are converted to $i\leftrightarrow c$, where $i\in V(G_1)$ and $j\in V(G_2)\setminus C$. As $i\leftrightarrow j$ and $i\leftrightarrow c$ have the same edges within $G_1$, they do not change the degree of any edge within $G_1$. So, the degree of each edge in $G_1$ is 
the same in both $\sigma^u$ and $\rho_{G_1}(\sigma^{u_{G_1}})$ and hence is the same in $\rho_{G_1}(\sigma^{v_{G_1}})$.

Now all we need to show is that the power of each $a_i$ is the 
same in both $\rho_{G_1}(\sigma^{u_{G_1}})$ and $\rho_{G_1}(\sigma^{v_{G_1}})$ for each $i\in G_1$. We observe that for every vertex $i\in V(G_1)\setminus \{c\}$, the number of treks in $\rho_{G_1}(\sigma^{u_{G_1}})$ with source $a_i$ remain the same as that in $\sigma^u$. The only change that can occur in $\rho_{G_1}(\sigma^{u_{G_1}})$ is the number of treks with source $c$. 
There are four types of treks in which $c$ can be the source: 
\begin{enumerate}[i)]
    \item treks of the form $c\leftrightarrow i$, where $i\in V(G_1)\setminus C$,
    \item treks of the form $c\leftrightarrow j$, where $j\in V(G_2)\setminus C$,
    \item treks of the form $i\leftrightarrow j$, where $i \in V(G_1)\setminus C$ and $j\in V(G_2)\setminus C$,
    \item $c\leftrightarrow c_i, c_i \in C$.
\end{enumerate}

\textbf{Case I} : The source of each trek of the form $i\leftrightarrow j$ with $i\in G_1$ and $j\in G_2$ lies in $G_1$ :

In this case, the treks of the form $(i)$ and $(iv)$ remain as it is whereas the treks of the form $(ii)$ and $(iii)$ are converted into $c\leftrightarrow c$ and $c\leftrightarrow i$ respectively, keeping the source of the treks as $c$. As all the sources lie within $G_1$, there are no treks of the form $i\leftrightarrow j, i\in G_1, j\in V(G_2)\setminus C$ with source in $G_2$ which could increase the power of $a_c$ in the image. Hence, the power of $a_c$ is preserved.

\textbf{Case II} : The source of each trek of the form $i\leftrightarrow j$ with $i\in G_1$ and $j\in G_2$ lies in $G_2$ :

In this case, the existing treks with source $c$ continue to contribute to the power of $a_c$ as in Case I. But,
there is a possibility of increasing the power of $a_c$ in $\rho_{G_1}(\sigma^{u_{G_1}})$ as the treks of the form $i\leftrightarrow j,i\in V(G_1),j\in V(G_2)\setminus C$ with source in $V(G_2)\setminus C$ are converted to $c\leftrightarrow i$ with source $c$. So, we need to show here that the increase in the power of $a_c$ remains the same in both $\rho_{G_1}(\sigma^{u_{G_1}})$ and $\rho_{G_1}(\sigma^{v_{G_1}})$.

We count the number of indeterminates of the form $\lambda_{dc}$ ( i.e, $d<c$ ) in the image of $\sigma^u$. This precisely gives us the number of the treks of the form $i\leftrightarrow j,i\in V(G_1),j\in V(G_2)\setminus C$ with source in $V(G_2)\setminus C$. This is because of the fact that if $\lambda_{ic}$ occurs in the image of $\sigma^u$ with $i\in V(G_1)\setminus C$, then it would imply that $\sigma^u$ has a trek which has an edge $i\rightarrow c,i\in V(G_1)$. This would mean the of treks of the form $i\leftrightarrow j, i\in V(G_1),j\in V(G_2)\setminus C$ cannot have source in $G_2$. As the number of indeterminates of the form $\lambda_{dc}$ is the same in both $\sigma^u$ and $\sigma^v$, we can conclude that the increase in the power of $a_c$ remains the same in $\rho_{G_1}(\sigma^{u_{G_1}})$ and $\rho_{G_1}(\sigma^{v_{G_1}})$.

So, $\psi_{G_1}(\rho_{G_1}(\sigma^{u_{G_1}}))-\psi_{G_1}(\rho_{G_1}(\sigma^{v_{G_1}}))=0$.
\end{proof}

\begin{defn}
Let $G$ be a safe gluing of $G_1$ and $G_2$ with $ST_{G_1}=CI_{G_1}$ and $ST_{G_2}=CI_{G_2}$. Then the \textit{lifting} of any binomial $f=\sigma_{i'j'}\sigma_{k'l'}-\sigma_{i'l'}\sigma_{k'j'} \in CI_{G_1}$ is defined as the set of binomials having the following form :
\[
\text{lift} (f)= 
\begin{cases} \sigma_{i'j'}\sigma_{k'l'}-\sigma_{i'l'}\sigma_{k'j'} & i',j',k',l' \in V(G_1)\setminus \{c\} \\
\sigma_{i'j}\sigma_{k'l'}-\sigma_{i'l'}\sigma_{k'j} & j'=c \text{ and for any } j\in V(G_2)\setminus D \text{ with }  i' \leftrightarrow c \subseteq i'\leftrightarrow j   \\
\sigma_{i'j'}\sigma_{k'l}-\sigma_{i'l}\sigma_{k'j'} & l'=c \text{ and for any } l\in V(G_2)\setminus D \text{ with } k' \leftrightarrow c \subseteq k'\leftrightarrow l  \\
\sigma_{i'p}\sigma_{ql'}-\sigma_{i'l'}\sigma_{pq} & j'=k'=c \text{ and for any } p, q \in V(G_2)\setminus D \text{ with } \\
\textcolor{white}{a} &  i' \leftrightarrow c \subseteq i'\leftrightarrow p,  c \leftrightarrow l' \subseteq q\leftrightarrow l' \text{ and } c \in p\leftrightarrow q
\end{cases}
\]
\end{defn}

We can similarly define the lift operation for binomials in $CI_{G_2}$. From the definition above, lift$(f)$ is not necessarily unique and can be lifted to multiple binomials. The lift operation can be seen as an inverse of the map $\rho_{G_1}$ (or $\rho_{G_2}$, although the $\rho_{G_i}$ maps are not invertible). In the next lemma, we show that the set of all binomials in lift($f$) lies in $CI_G$ and also in $ST_G$ for any $f= \sigma_{i'j'}\sigma_{k'l'}-\sigma_{i'l'}\sigma_{k'j'} \in CI_{G_1}$.

\begin{lemma}\label{lemma:lifted binomials in ST_G}
Let $f$ be any binomial in $CI_{G_1}$ of the form $\sigma_{i'j'}\sigma_{k'l'}-\sigma_{i'l'}\sigma_{k'j'} \in CI_{G_1}$. Then the set of all the binomials in lift$(f)$ lies in both $CI_G$ and $ST_G$.
\end{lemma}
\begin{proof}

\begin{enumerate}[i)]
\item 
We first show that lift($f$) $\in CI_G$ for all the four cases given in the definition of lift.

\begin{enumerate}[a)]
\item 
In the first case, as $CI_{G_1}\subseteq CI_G$, $\sigma_{i'j'}\sigma_{k'l'}-\sigma_{i'l'}\sigma_{k'j'} \in CI_{G}$ when $i',j',k',l' \in V(G_1)\setminus \{c\}$.

\item 
When $j'=c$ and $i',k',l'\in V(G_1)\setminus \{c\}$, then $f\in CI_{G_1}$ implies that $\{l'\}$ $d$-separates $\{i',k'\}$ from $\{c\}$ (or $\{i'\}$ $d$-separates $\{k'\}$ from $\{l',c\}$). Now, as every trek from $i'$ and $k'$ to any vertex in $V(G_2)\setminus C$ passes through $\{c\}$, we can conclude that $\{l'\}$ $d$-separates $\{i',k'\}$ from $V(G_2)\setminus D$. So, $\sigma_{i'j}\sigma_{k'l'}-\sigma_{i'l'}\sigma_{k'j} \in CI_{G}$ for any $j\in V(G_2)\setminus D$. (Similar argument follows when $\{i'\}$ $d$-separates $\{k'\}$ from $\{l',c\}$.)

\item 
A similar argument as in (b) follows here.

\item 
When $j'=k'=c$ and $c\in p\leftrightarrow q$, then we know that every trek from $i'$ to $q$ passes through $c$. Similarly, every trek from $l'$ to $p$ passes through $c$. Further, as $\sigma_{i'c}\sigma_{cl'}-\sigma_{i'l'}\sigma_{cc} \in CI_{G_1}$, we know that $\{c\}$ $d$-separates $\{i\}'$ from $\{l'\}$. From the definition of lift, we know that $c$ lies in $p\leftrightarrow q$. But as $CI_{G_2}=SP_{G_2}$, we can also say that $\{c\}$ $d$-separates $\{p\}$ from $\{q\}$. Combining all the separations, we have that $\{c\}$ $d$-separates $\{i',p\}$ from $\{l',q\}$ and hence $\sigma_{i'p}\sigma_{ql'}-\sigma_{i'l'}\sigma_{pq} \in CI_G$.
\end{enumerate}

\item 
In each case above, the $d$-separation criterion forces all the four shortest treks of each binomial to pass through a particular vertex. So, a decomposition similar to the one shown in the proof of Lemma \ref{lemma:invalid move} is always possible and hence lift($f$)$\in ST_G$ for all the four cases.
\end{enumerate}
\end{proof}

\begin{lemma}\label{lemma:match rho1}
Let $G$ be a safe gluing of $G_1$ and $G_2$, with the map $\rho_{G_1}$ defined as above.
Suppose that $\sigma^u - \sigma^v \in ST_G$ 
and this binomial only involves $\sigma_{ij}$ indeterminates corresponding to treks.
Suppose that $ST_{G_1} =  CI_{G_1}$.
Then, there is a set of quadratic moves in $CI_G$ that will transform
$\sigma^u$ into a monomial $\sigma^{u'}$ such that
$\rho_{G_1}(\sigma^{u'})  =  \rho_{G_1}(  \sigma^v)$.
\end{lemma}

\begin{proof}
Since $ST_{G_1}$ equals $CI_{G_1}$, 
by Theorem \ref{thm:connected-generators} we know that 
either $\rho_{G_1}(\sigma^{u_{G_1}})$ is equal to $\rho_{G_1}(\sigma^{v_{G_1}})$ or we can reach 
from $\rho_{G_1}(\sigma^{u_{G_1}})$ to $\rho_{G_1}(\sigma^{v_{G_1}})$ by making a finite set 
of moves from the set of $2\times 2$ minors of 
$\Sigma_{A\cup C,B\cup C}$, 
for all possible $d$-separations of $G_1$.

By using the map $\rho_{G_1}$ 
we lift each move each move $\sigma_{i'j'}\sigma_{k'l'}-\sigma_{i'l'}\sigma_{k'j'}$ 
in $G_1$ to a corresponding move $\sigma_{ij}\sigma_{kl}-\sigma_{il}\sigma_{kj}$ 
in $G$, where  
\[
\rho_{G_1}(\sigma_{ij})=\sigma_{i'j'}, \hspace{.2cm} \rho_{G_1}(\sigma_{kl})=\sigma_{k'l'}, \hspace{.2cm} \rho_{G_1}(\sigma_{il})=\sigma_{i'l'} \text{ and }\hspace{.1cm}  \rho_{G_1}(\sigma_{kj})=\sigma_{k'j'}.
\]
These moves take $\sigma^u$ to $\sigma^{u'}$ for some $u'$ such that $\sigma^{u'}$ and $\sigma^v$ have the same subgraph within $G_1$.
\end{proof}

We illustrate the technique used in the proof with an example.
\begin{ex}
Let $G=\{1\rightarrow 2, 1\rightarrow 4, 1\rightarrow 6, 1\rightarrow 8, 2\rightarrow 3, 4\rightarrow 5, 6\rightarrow 7, 8\rightarrow 9 \}$ be a DAG with $V(G_1)=\{1,2,3,6,7\}$ and $V(G_2)=\{1,4,5,8,9\}$. Let 
\[
f= \sigma^u-\sigma^v=\sigma_{56}\sigma_{47}\sigma_{67}\sigma_{28}-\sigma_{66}\sigma_{27}\sigma_{57}\sigma_{48}\in ST_G.
\]
Then $\rho_{G_1}(\sigma^{u_{G_1}})= \sigma_{16}\sigma_{17}\sigma_{67}\sigma_{12}$. We take
\[
m_1= \sigma_{16}\sigma_{67}-\sigma_{66} \sigma_{17}  \in CI_{G_1} 
\]
as the first move which takes $\rho_{G_1}(\sigma^{u_{G_1}})$ to $\sigma_{66}\sigma_{17}^2\sigma_{12}$. As
\[
\rho_{G_1}(\sigma_{56}\sigma_{67}-\sigma_{66} \sigma_{57})=\sigma_{16}\sigma_{67}-\sigma_{66} \sigma_{17},
\]
we lift $m_1$ to $m_1'=\sigma_{56}\sigma_{67}-\sigma_{66} \sigma_{57} \in CI_G$. Now, we take
\[
m_2= \sigma_{17}\sigma_{12}-\sigma_{27} \sigma_{11}  \in CI_{G_1} 
\]
as the second move which takes $\sigma_{66}\sigma_{17}^2\sigma_{12}$ to $\sigma_{66}\sigma_{17}\sigma_{27}\sigma_{11}$. Further, as
\[
\rho_{G_1}(\sigma_{47}\sigma_{28}-\sigma_{27} \sigma_{48})=\sigma_{17}\sigma_{12}-\sigma_{27} \sigma_{11},
\]
we lift $m_2$ to $m_2'=\sigma_{47}\sigma_{28}-\sigma_{27} \sigma_{48}$. Observe that applying $m_1'$ and then $m_2'$ on $\sigma^u$ takes $\sigma^u$ to $\sigma^v$.
\end{ex}

In a similar way, we can define the map $\rho_{G_2}$ 
and get a set of moves which would take $\rho_{G_2}(\sigma^{u'})$ to $\rho_{G_2}(\sigma^{v})$.
This in turn would give us a corresponding set of moves in $G$ which would take 
$\sigma^{u'}$ to $\sigma^{v'}$ for some $v'$ such that $\sigma^{v'}$ and $\sigma^v$ 
have the same subgraph within $G_2$. But before that, it is important to check that the second set of lifted moves obtained from $\rho_{G_2}$ does not affect the structure of $\sigma^{u'}$ within $G_1$.

\begin{prop} \label{prop:similarity preserved}
Let $m=\sigma_{ij}\sigma_{kl}-\sigma_{il}\sigma_{kj}$ be a move obtained as a lift of one of the moves in $CI_{G_2}$ which takes $\rho_{G_2}(\sigma^{u'})$ closer to $\rho_{G_2}(\sigma^v)$. Then $\rho_{G_1}(\sigma^{u'})=\rho_{G_1}(m(\sigma^{u'}))$. 
\end{prop}

\begin{proof}
As $\rho_{G_1}(\sigma^{u'})=\rho_{G_1}(\sigma^v)$, the move $m$ corresponds to a $d$-separation by a vertex in $V(G_2)\setminus C$. Let that vertex be $c'$. Now, if $i,j,k,l \in V(G_2)\setminus C$, then clearly $m$ does not affect the structure of $\sigma^{u'}$. 
So, let $i,k \in V(G_1)\setminus C$ and $ j, l \in V(G_2)\setminus C$. Then we have
\begin{eqnarray*}
i\leftrightarrow j &=&  i \leftrightarrow \cup \ c\leftrightarrow c' \cup c'\leftrightarrow j \\
k\leftrightarrow l &=&  k \leftrightarrow \cup \ c\leftrightarrow c' \cup c'\leftrightarrow l \\
i\leftrightarrow l &=&  i \leftrightarrow \cup \ c\leftrightarrow c' \cup c'\leftrightarrow l \\
k\leftrightarrow j &=&  k \leftrightarrow \cup \ c\leftrightarrow c' \cup c'\leftrightarrow j.
\end{eqnarray*}
This gives us that the multi-digraph of both $\rho_{G_1}(\sigma_{ij}\sigma_{kl})$ and $\rho_{G_1}(\sigma_{il}\sigma_{kj})$ are same. So, we can conclude that $m$ does not affect the structure of $\rho_{G_1}(\sigma^{u'})$ and hence $\rho_{G_1}(\sigma^{u'})=\rho_{G_1}(m(\sigma^{u'}))$.
\end{proof}

\begin{lemma}\label{lemma:final binomials}
Let $G$ be a safe gluing of $G_1$ and $G_2$, with the maps 
$\rho_{G_1}$ and $\rho_{G_2}$ defined as above.
Suppose that $\sigma^u - \sigma^v \in ST_G$ 
and this binomial only involves $\sigma_{ij}$ indeterminates corresponding to treks.
Suppose that $\rho_{G_1}( \sigma^u) = \rho_{G_1}( \sigma^v)$ and 
$\rho_{G_2}( \sigma^u) = \rho_{G_2}( \sigma^v)$.
Then $\sigma^u$ and $\sigma^v$ can be connected by quadratic binomials
in $CI_G$. 
\end{lemma}

\begin{proof}
We can assume that $\sigma^u$ and $\sigma^v$ have no indeterminates in common.
Since $\sigma^u$ and $\sigma^v$ have the same image under $\rho_{G_1}$ and $\rho_{G_2}$
this implies
 that we cannot have any indeterminates of the form 
 $\sigma_{ij},i,j\in V(G_1)\setminus \{c\}$ or $i,j\in V(G_2)\setminus \{c\}$ 
 in the monomial factors. This is because the indeterminates of this 
 form are mapped to itself by either of the two maps which would mean 
 that $\sigma^{u}$ and $\sigma^v$ would still have some more common factors between them. 
So, all the indeterminates appearing in the two factors need to 
contain $c$ as an end point or as a vertex in their corresponding 
shortest treks and both end points not lying within the same subgraph (i.e, $G_1$ or $G_2$).

Consider an arbitrary trek $i\leftrightarrow j$ in 
$\sigma^{u}$ which is not present in $\sigma^v$. 
We select the trek in $\sigma^v$ which has the highest number 
of common edges with $i\leftrightarrow j$. 
Let that trek be $i'\leftrightarrow j'$ and let 
$s\leftrightarrow t $ be the common trek in both the treks. 
Let $s_1$ and $t_1$ be the vertices adjacent to 
$s$ and $t$ respectively in $i\leftrightarrow j$. Similarly, 
let $s'$ and $t'$ be the vertices adjacent to $s$ and $t$ 
respectively in $i'\leftrightarrow j'$. 
Let $p$ be the vertex in $s\leftrightarrow t$ 
adjacent to $t$ (see Figure \ref{fig:path moves} for an illustration of the idea).

As $\psi_G(\sigma^{v'}-\sigma^v)=0$ 
there must exist a path $x\leftrightarrow y$ in 
$\sigma^v$ containing the edge $t\leftrightarrow t_1$. 
We know that all the indeterminates appearing in both the 
monomial factors need to contain $c$. 
This implies that $c$ must lie within the common trek $s\leftrightarrow t$. 
Let $i,i'$ and $x$ be in $V(G_1)\setminus C$ and $j,j',y$ be in $V(G_2)\setminus C$. 
The move $m=\sigma_{i'j'}\sigma_{xy}-\sigma_{i'y}\sigma_{xj'}$ 
is now a valid move as none of the vertices in $i'\leftrightarrow p$ can have a shorter connection to any vertex in $t_1\leftrightarrow y$ (as every shortest trek from a vertex in $V(G_1)\setminus C$ to $V(G_2)\setminus C$ must pass through $c$).

Applying $m$ on $\sigma^v$ increases the length of the 
common trek between $i\leftrightarrow j$ and $i'\leftrightarrow j'$ by at least 1. 
As any move preserves the kernel of $\psi_G$, $m(\sigma^{u})- \sigma^v$ 
still lies in $ST_G$. Repeating this process again, we can continue to shorten the length of
the disagreement until the resulting monomials are the same. 
\end{proof}

\begin{figure}
\begin{tikzpicture}
\filldraw[black]
(0,0) circle [radius=.04] node [below] {$i'$}
(1,0) circle [radius=.02] 
(1.15,0) circle [radius=.02] 
(1.3,0) circle [radius=.02] 
(1.45,0) circle [radius=.02]
(2,0) circle [radius=.02] node [below] {$s'$}
(3,0) circle [radius=.02] node [below] {$s$}
(3.15,0) circle [radius=.02] 
(3.3,0) circle [radius=.02]  
(3.45,0) circle [radius=.02] 
(4,0) circle [radius=.02] node [below] {$p$}
(5,0) circle [radius=.02] node [below] {$t$}
(6,0) circle [radius=.02] node [below] {$t'$}
(6.15,0) circle [radius=.02] 
(6.3,0) circle [radius=.02] 
(6.45,0) circle [radius=.02] 
(6.6,0) circle [radius=.02]
(7,0) circle [radius=.04] node [below] {$j'$}
(4,-1) circle [radius=0] node [below] {$\sigma^{v}$}
(2,1.75) circle [radius=.04] node [above] {$x$}
(4,.25) circle [radius=.02]
(5,.25) circle [radius=.02]
(5.66,.87) circle [radius=.03] node [below] {$t_1$}
(6.6,1.75) circle [radius=.04] node [above] {$y$}
(2.28,1.34) circle [radius=.02]
(2.587,.87) circle [radius=.03] 
(2.4,1.15) circle [radius=.02]
(5.95,1.15) circle [radius=.02]
(6.15,1.34) circle [radius=.02]

(8,0) circle [radius=.04] node [below] {$i$}
(9,0) circle [radius=.02] 
(9.15,0) circle [radius=.02] 
(9.3,0) circle [radius=.02] 
(9.45,0) circle [radius=.02]
(10,0) circle [radius=.02] node [below] {$s_1$}
(11,0) circle [radius=.02] node [below] {$s$}
(11.15,0) circle [radius=.02] 
(11.3,0) circle [radius=.02]  
(11.45,0) circle [radius=.02] 
(12,0) circle [radius=.02] node [below] {$p$}
(13,0) circle [radius=.02] node [below] {$t$}
(14,0) circle [radius=.02] node [below] {$t_1$}
(14.15,0) circle [radius=.02] 
(14.3,0) circle [radius=.02] 
(14.45,0) circle [radius=.02] 
(14.6,0) circle [radius=.02]
(15,0) circle [radius=.04] node [below] {$j$}
(3,.25) circle [radius=.02]
(11.5,-1) circle [radius=0] node [below] {$\sigma^{v'}$};

\draw
(0,0)--++(1,0)
(1.45,0)--++(.55,0)
(2,0)--++(1.15,0)
(3.45,0)--++(2.7,0)
(6.6,0)--++(.4,0)
(2,1.75)--++(1,-1.5)
(3,.25)--++(2,0)
(5,.25)--++(1.6,1.5)

(8,0)--++(1,0)
(9.45,0)--++(.55,0)
(10,0)--++(1.15,0)
(11.45,0)--++(2.7,0)
(14.6,0)--++(.4,0);

\end{tikzpicture}
\caption{Graphs of $\sigma^v$ and $\sigma^{v'}$. We use undirected treks in the figure to represent treks of unknown direction as the proof is independent of the direction of the treks.\label{fig:path moves}}
\end{figure}
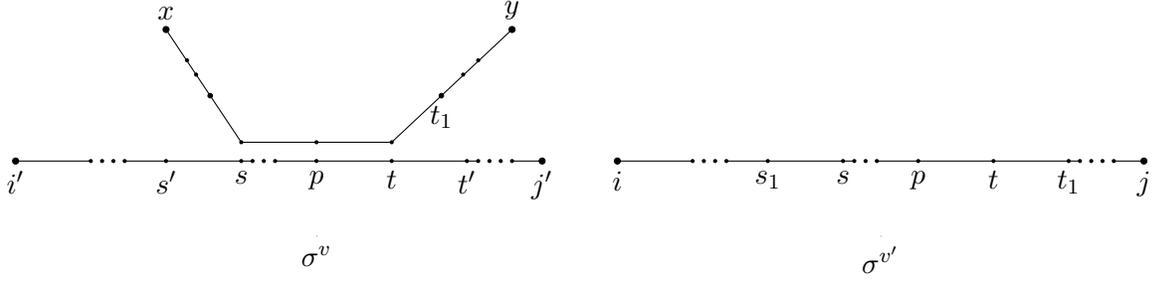

Using all the results and observations that we have so far, we give a proof of 
the main result of this section, which shows that quadratic generation of the 
shortest trek ideals is preserved under the safe gluing operation.

\begin{thm}\label{thm:ker(psi_G)=CI_G}
Let $G_1$ and $G_2$ be two DAGs such that $ST_{G_1}=CI_{G_1}$ and $ST_{G_2}=CI_{G_2}$. If $G$ is the DAG obtained after a safe gluing of $G_1$ and $G_2$ at an $N$-clique, then $ST_G$ is equal to both $CI_G$ and $I_G$, and hence $I_G$ is toric and is generated by quadratic binomials. 
\end{thm}

\begin{proof}[Proof of Theorem \ref{thm:ker(psi_G)=CI_G}]
Let $\sigma^u-\sigma^v$ be an arbitrary binomial in $ST_G$. Then in order to prove that 
$ST_G=CI_G$, we need to show that $\sigma^u$ and $\sigma^v$ are 
connected by the moves in $\mathcal{F}$, where $\mathcal{F}$ is the set of all the generators of $CI_G$. 

Lemma \ref{lemma:match rho1} shows that we can apply quadratic moves in $CI_{G}$
to transform $\sigma^u$ into a monomial $\sigma^{u'}$ such that $\rho_{G_1}(\sigma^{u'}) = 
\rho_{G_1}(\sigma^{v})$.  Applying the analogous result for $G_2$, we see that we
can apply quadratic moves in $CI_G$ to transform $\sigma^{u'}$ into $\sigma^{v'}$
such that $\rho_{G_1}(\sigma^{v'}) = \rho_{G_1}(\sigma^{v})$  and  
$\rho_{G_2}(\sigma^{v'}) = \rho_{G_2}(\sigma^{v})$.
Then applying Lemma \ref{lemma:final binomials}, we see that 
$\sigma^{v'}$ and $\sigma^v$ can be connected using binomials in $CI_G$. 
This shows that $ST_G \subseteq CI_G  \subseteq I_G$.   
But as $I_G$ and $ST_G$ are both prime ideals of the same dimension, this shows that
all three ideals are equal. Now that we have the equality that $I_G=CI_G$ and is also toric, we can conclude that $I_G$ is generated by the $2 \times 2$ minors of $\Sigma_{\{A\cup C, B\cup C\}}$ for all possible $d$-separations of $G$ with $|C|=1$.
\end{proof}


\section{Conjectures} \label{sec:conjectures}

We close the paper by giving some conjectures about the Gaussian DAGs with toric vanishing ideals.
These include some main conjectures, and also conjectures of a more technical nature
that would be important tools for proving the main conjectures.
We also discuss some consequences of these auxiliary conjectures.

Our first main conjecture relates goes with a running theme throughout the paper,
identifying the underlying combinatorics of the toric structure when
$I_G$ is actually a toric ideal.

\begin{conj}\label{conj:toric iff I_G=ST_G}
A DAG $G$ has a toric vanishing ideal if and only if $I_G=ST_G$.
\end{conj}

Note, as mentioned previously, there are DAGs $G$ such that $ST_G$ exists, but it not equal to $I_G$.
Our second main conjecture concerns the combinatorial construction of the DAGs for which
$I_G$ is toric.

\begin{conj}\label{conj: toric implies safe gluing}
If $G$ is a DAG such that $I_G$ is toric, 
then either:
\begin{enumerate}
\item   $G$ is a complete DAG,
\item  $G$ is either a safe gluing or the gluing at sinks of two smaller DAGs that also have toric vanishing ideals, or
\item $G$ is obtained by adding a sink to a smaller DAG which has a toric vanishing ideal.
\end{enumerate}
\end{conj}

Important auxiliary conjectures that we have seen so far in the paper
concern the safe gluing operation, in particular, Conjecture \ref{conj:safe gluing conjecture},
that safe gluing preserves the property of $I_G$ being equal to $ST_G$.
Another conjecture that seems key to proving 
classification results for toric vanishing ideals is the 
following conjecture, that would rule out many graphs from having toric vanishing ideals.

\begin{conj}\label{conj:higher d-separation}
Let $G$ be a DAG and $i,j$ be two vertices in $G$ such that the minimal size of a d-separating set of $i$ and $j$
is $2$ or larger. Then $I_G$ is not toric.
\end{conj}

Assuming the conjecture is true, we have two results on when the vanishing ideal is not toric.

\begin{lemma}\label{lemma: collider and non collider}
Suppose that Conjecture \ref{conj:higher d-separation} is true.
Let $G$ be a DAG and $i,j$ be two vertices in $G$ having 
at least 2 different paths $P_1$ and $P_2$ between them. 
If $P_2$ is a trek containing the vertex $c$ and $P_1$ is a path having exactly one collider at $c$, 
then $I_G$ is not toric if Conjecture \ref{conj:higher d-separation} is true. 
\end{lemma}

\begin{proof}
\textbf{Case I:} $P_1$ and $P_2$ have no common vertices except $i, c$ and $j$ :

The proof follows from the $d$-separation of $i$ and $j$. As $c$ is the only collider within $P_1$, any set $C$ which contains $c$ and $d$-separates $i$ from $j$ has to contain at least one more vertex from $P_1$. This is because $C=\{c\}$ is not enough to $d$-separate $i$ and $j$. Hence, by using Conjecture \ref{conj:higher d-separation} we can conclude that $I_G$ is not toric.

\textbf{Case II:} $P_1$ and $P_2$ have more than 3 common vertices :

Let $i_1$ be the last common vertex before $c$ and $j_1$ be the first common vertex after $c$ within the two paths. Then following Case I by replacing $i$ and $j$ with $i_1$ and $j_1$ respectively completes the proof.
\end{proof}

\begin{lemma}\label{lemma: no shortest trek}
Suppose that Conjecture \ref{conj:higher d-separation} is true.
Let $G$ be a DAG where the shortest trek map cannot be defined. 
Then $I_G$ is not toric.
\end{lemma}

\begin{proof}
The shortest trek map in $G$ is not defined when there is no unique shortest trek between two vertices. 
Let $i$ and $j$ be two vertices in $G$ having two treks 
$P_1$ and $P_2$ between them of the same length and have no other trek whose length is smaller.

\textbf{Case I:} There is no common vertex between $P_1$ and $P_2$ except $i$ and $j$.  
In this case, we will have to select at least one vertex from each of the two treks 
to $d$-separate $i$ and $j$. Hence by Conjecture \ref{conj:higher d-separation} 
we can conclude that $I_G$ is not toric.

\textbf{Case II:} Suppose that $P_1$ and $P_2$ have at least one common vertex.
Without loss of generality, we can assume that $i<j$. Let $c$ be the first common vertex between $P_1$ and $P_2$. Then the treks $P_1$ and $P_2$ can be written as 
\begin{eqnarray*}
P_1&=&P_1(i\rightleftharpoons c)\cup P_1(c\rightleftharpoons j) \text{ and} \\
P_2&=&P_2(i\rightleftharpoons c)\cup P_2(c\rightleftharpoons j),
\end{eqnarray*}
where $P_1(i\rightleftharpoons c)$ and $P_2(i\rightleftharpoons c)$ denote the 
trek between $i$ and $c$ within the treks $P_1$ and $P_2$ respectively.
Let the lengths of $P_1(i\rightleftharpoons c),P_1(c\rightleftharpoons j), P_2(i\rightleftharpoons c)$ and $P_2(c\rightleftharpoons j)$ be $r_1,s_1,r_2$ and $s_2$ respectively. Then we have
\begin{equation}\label{equation:equality}
r_1+s_1=r_2+s_2.
\end{equation}
This gives us two new paths between $i$ and $j$, namely 
$P_3=P_1(i\rightleftharpoons c)\cup P_2(c\rightleftharpoons j)$ and 
$P_4=P_2(i\rightleftharpoons c)\cup P_1(c\rightleftharpoons j)$. 
If either of $P_3$ or $P_4$ has a collider at $c$, 
then by Lemma \ref{lemma: collider and non collider} we know that $I_G$ is not toric. 
So, we can assume that $P_3$ and $P_4$ are also treks. 

Now, let  $r_1<r_2$. Then by equation \ref{equation:equality}, we know that $s_2<s_1$. From these inequalities, we get that the trek $P_3$ is of length $r_1+s_2$ which is smaller than $r_1+s_1$, a contradiction. ( Similar argument follows for $r_2<r_1$ ).
Thus, we have $r_1=r_2$ and $s_1=s_2$. Now replacing $j$ with $c$, we can follow the same argument as that in Case I. Hence, $I_G$ is not toric. 
\end{proof}

Recall that an undirected graph is \emph{chordal} if it has no induced cycles of length $\geq 4$.
For the remainder of the section, we consider DAGs $G$ whose undirected version $G^\sim$ is 
a chordal graph.  
In Theorem \ref{thm:safe gluing at n clique} we used the condition that $I_{G_1}$ and $I_{G_2}$ 
can have at most one common indeterminate $\sigma_{cc}$. 
In the next Lemma we show that if Conjecture \ref{conj:higher d-separation} is true, 
then the above condition of Theorem \ref{thm:safe gluing at n clique} 
is satisfied when at least one of $G_1$ or $G_2$ 
is a chordal DAG.  
So this provides further evidence in favor of
Conjecture \ref{conj:safe gluing conjecture}. 

\begin{lemma}
Suppose that Conjecture \ref{conj:higher d-separation} is true.
Let $G_1$ and $G_2$ be two DAGs with $I_{G_1} = ST_{G_1}$ and  
$I_{G_2} = ST_{G_2}$.  Let 
$G$ be the resultant DAG obtained after a safe gluing of $G_1$ and $G_2$ at an $N$-clique. 
Let $C = \{c\} \cup D$ be the vertices in the $N$-clique where $c$ is the choke point.
Let $c' \in C$ and $d \in D$.
If $G_1$ is chordal and $p_1$ is a vertex in $G_1\setminus C$ such that the shortest trek 
$p_1\leftrightarrow c'$ contains the edge $c'\rightarrow d$ then $G$ can be constructed by safe gluing two DAGs 
 at an $(N-1)$-clique. 
\end{lemma}

\begin{proof}
Let $p_1-p_2- \cdots- p_m - c'\rightarrow d$ be the shortest trek between 
$p_1$ and $d$, where $p_1-p_2$ denotes the edge between $p_1$ and $p_2$ of unknown direction. 
Then $p_m- c' \rightarrow d$ is also the shortest trek between $p_m$ and $d$. 
Let us assume that $G$ cannot be constructed by safe gluing two
DAGs at an $(N-1)$-clique. 
Then there must exist another path from $p_m$ to $d$ not containing the edge 
$c'\rightarrow d$. We select that path whose vertices are adjacent to either 
$p_m,c'$ or $d$. 
Let $p_m-q_1-\cdots-q_r\rightarrow d$ be such a path.  
As $G_1$ is chordal, either $p_m\rightarrow d$ is an edge or there exists an edge between 
$q_r$ and $c'$. If $p_m\rightarrow d$ is an edge, then $p_m\rightarrow d$ becomes a shorter trek than 
$p_m- c' \rightarrow d$, which is a contradiction. 
If there is an edge between $q_r$ and $c'$, 
there must also be an edge between $p_m$ and $q_r$ 
(again as $G_1$ is chordal and $I_{G_1}$ is toric). 
Independent of the direction of these two edges $q_r-c'$ and $p_m-q_r$, 
we can say that $p_m$ is $d$-separated from $d$ by at least two vertices $c'$ and $q_r$. 
Thus by using Conjecture \ref{conj:higher d-separation} we can conclude that 
$I_{G_1}$ is not toric, which is a contradiction.
\end{proof}

So far we have shown that safe gluing preserves the toric property of the vanishing ideals. 
But it is interesting to check if a DAG $G$ with toric vanishing ideal 
can always be obtained as a safe gluing of smaller DAGs with toric vanishing ideals. 
We end this paper with the conjecture that such a decomposition always exist for chordal graphs if Conjecture \ref{conj:higher d-separation} is true.

\begin{conj}\label{conj:decompose G}
Suppose that Conjecture \ref{conj:higher d-separation} is true.
Let $G$ be a chordal DAG with toric vanishing ideal. 
Then there exist $G_1$ and $G_2$ with toric vanishing ideals such that 
$G$ can be obtained as a safe gluing of $G_1$ and $G_2$ at an $N$-clique.
\end{conj}


\section*{Acknowledgments}

Pratik Misra and Seth Sullivant were partially supported by the 
US National Science Foundation (DMS 1615660).



\begin{thebibliography}{99}

\bibitem{Graphical Models}
S. Lauritzen.
\emph{Graphical Models}. \emph{Oxford Statistical Science Series} {\bf{17}} Clarendon Press, Oxford, 1996.


\bibitem{Maathuis2019}
M. Maathuis, M. Drton, S. Lauritzen and M. Wainwright. (Editors)
\emph{Handbook of Graphical Models.}
 Chapman and Hall/CRC Handbooks of Modern Statistical Methods. CRC Press, Boca Raton, FL, 2019.

\bibitem{Block graphs}
P. Misra, S. Sullivant.
Gaussian graphical models with toric vanishing ideals. \emph{Annals of the Institute of Statistical Mathematics},
DOI 10.1007/s10463-020-00765-0, 2020.

\bibitem{Ancestral graph}
T.S. Richardson, P. Spirtes. Ancestral graph Markov models. \textit{The Annals of Statistics} {\bf{30}}, 962-1030, 2002.

\bibitem{Causation}
P. Spirtes, C. Glymour, R. Scheines, D. Heckerman. \emph{Causation, Prediction, and Search}. MIT Press, Cambridge, MA, 2000.

\bibitem{GB and CP}
B. Sturmfels.
\emph{Gr{\"o}bner Bases and Convex Polytopes}.
University Lecture Series {\bf{8}}, AMS, 1996.

\bibitem{BS n CU}
B. Sturmfels, C. Uhler.
Multivariate Gaussians, semidefinite matrix completion and convex algebraic geometry. \emph{Annals of the Institute of Statistical Mathematics} {\bf{62}}, 603-638, 2010.

\bibitem{Gaussian Networks}
S. Sullivant. Algebraic geometry of Gaussian Bayesian networks. \emph{Advances in Applied Mathematics} {\bf{40}}, Issue 4, 482-513, May 2008.

\bibitem{Algebraic Statistics}
S. Sullivant.
\emph{Algebraic Statistics}. Volume {\bf{194}} of \emph{Graduate Studies in Mathematics}. AMS, Providence, RI, 2018.

\bibitem{Trek Separation}
S. Sullivant, K. Talaska, J. Draisma.
Trek separation for Gaussian graphical models.
\emph{The Annals of Statistics} {\bf{38}}, No 3, 1665-1685, 2010.




\end{thebibliography}
\end{document}